\documentclass{amsart}
\usepackage[all]{xy}
\usepackage{amsmath}
\usepackage{stmaryrd}
\usepackage{amsfonts}
\usepackage{amssymb}
\usepackage{amsthm}
\usepackage{amscd}
\usepackage{latexsym}
\usepackage{txfonts}
\usepackage{float}
\usepackage{enumerate}
\usepackage{tabularx}
\usepackage{hyperref}
\theoremstyle{plain}
\newtheorem{theorem}{Theorem}[section]
\newtheorem{lemma}[theorem]{Lemma}
\newtheorem{cor}[theorem]{Corollary}

\newtheorem{proposition}[theorem]{Proposition}
\newtheorem{defi}[theorem]{Definition}
\newtheorem*{theorem*}{Theorem}
\theoremstyle{remark}\newtheorem{con}[theorem]{Construction}
\theoremstyle{remark} \newtheorem*{remark}{Remark}
\theoremstyle{remark} 
\newcommand{\bb}{\mathbb}
\newcommand{\mrm}{\mathrm}
\newcommand{\mbf}{\mathbf}
\newcommand{\ov}{\overline}

\newcommand{\mcal}{\mathcal}
\newcommand{\mb}{\mathbf}
\newcommand{\mfk}{\mathfrak}
\newcommand{\one}{{1}}

\newcommand{\sli}{\mathrm{SL}}

\newcommand{\gl}{\mathrm{GL}}
\newcommand{\gu}{\mathrm{GU}}
\newcommand{\bo}{\mathrm{B}}
 
\newcommand{\irr}{\mathrm{Irr}}
\newcommand{\F}{{\mathbb {F}}} 
\newcommand{\PP}{{\mathbb{P}}} 
\newcommand{\R}{{\mathbb {R}}}

\newcommand{\C}{{\mathbb {C}}}
\newcommand{\GL}{{\mathrm{GL}}}
\newcommand{\GU}{{\mathrm{GU}}}
\newcommand{\SL}{{\mathrm{SL}}}

\title{Perfect state transfer in graphs related to linear groups in two dimensions}

\author[V.R.T. Pantangi]{Venkata Raghu Tej Pantangi}
\address{Department of Mathematics and Statistics\\University of Regina\\Regina SK S4S 0A2\\Canada}
\author[P. Sin]{Peter Sin}
\address{Department of Mathematics\\University of Florida\\ P. O. Box 118105\\ Gainesville FL 32611\\ USA}

\thanks{This work was partially supported by a grant from the Simons Foundation (\#633214
 to Peter Sin). The first author was supported by the Pacific Institute of Mathematical Sciences Postdoctoral Fellowship.}
 
\begin{document}

\date{}

\begin{abstract}
We construct families of graphs from linear groups $\SL(2,q)$, $\GL(2,q)$
  and $\GU(2,q^2)$, where $q$ is an odd prime power, with the property that the continuous-time quantum walks on the associated networks of qubits admit perfect state transfer.
\end{abstract}
\maketitle
\section{Introduction} 
Let $X$ be a simple finite graph, with adjacency matrix $A$, for some
fixed ordering of the vertex set $V(X)$. In this paper we shall study a quantum mechanical model known as a {\it continuous-time quantum walk} on $X$, based on the $XX$-Hamiltonian \cite[IV.E]{Kay}. In this model, the evolution of states is governed by the unitary matrices $U(t)=\exp(-itA)$, for $t\in\R$ (denoting time). For background on quantum walks we refer the reader to \cite{Christandl} and \cite{CoutinhoGodsilBook}.

A phenomenon of central importance in the theory is {\it perfect  state transfer} (PST).
Let $\C^{V(X)}$ denote the Hilbert space  in which the characteristic vectors
$e_a$, $a\in V(X)$, form an orthonormal basis.
Let $a$ and $b$ be vertices of $X$. We say that we have perfect state transfer from $a$ to $b$ at time $\tau$ if $U(\tau)e_a=\gamma e_b$
for some $\gamma\in\C$ of absolute value 1; in other words, an initial state  concentrated on the vertex $a$ evolves at time $\tau$ to one concentrated on $b$.
A spectral characterization (c.f Theorem~2.4.4 \cite{coutinhothesis}) of graphs admitting PST was obtained in Coutinho's thesis \cite{coutinhothesis}. 
An early result of Godsil \cite[Corollary 10.2]{GodsilStateTransfer} shows that for a given maximal degree there are only finitely many connected graphs that admit PST. It is therefore of interest to consider infinite families of connected graphs that admit PST. In this paper we construct several such families.

In order to provide context for our results we give a brief review, based partly on the list in \cite{CoutinhoGodsil}, of previous work in this direction.
Among paths, only those with two or three
vertices admit PST. More generally, it has recently been shown \cite{CoutinhoJulianoSpier} that these paths are the only trees that admit PST. Complete graphs $K_n$ do not admit PST for $n>2$.
%On the other hand the complete bipartite graphs $K_{2,n}$ has PST
%between the two vertices of degree $n$.
Among strongly regular graphs, we have PST only in the trivial case of a
complement to a disjoint union of an even number of copies of $K_2$. However, bipartite
doubles of many
strongly regular graphs and distance regular graphs\cite{coutinho2015perfect} do provide examples of PST.

For graphs belonging to the Johnson scheme $J(n,k)$ of $k$-subsets in a set of size $n$, it was proved in \cite{Ahmadi} that PST can occur only if $n=2k$. In this case, the Kneser graph $K(2k,k)$ is a disjoint union of $\frac12\binom{2k}{k}$
copies of $K_2$ so, obviously, it admits PST. When $\binom{2k}{k}$ is divisible by $4$, the complement also admits PST. Apart from these, no other graph belonging to the scheme has yet been found that admits PST, and it has been shown that $K(2k,k)$ is the only generalized Johnson graph admitting PST.

There are some general results on the behavior of PST with respect to certain constructions such as products, covers and joins \cite{Angeles-Canul}, \cite{CoutinhoGodsil}, yielding further examples. There are also general results  about graphs belonging to association schemes \cite{coutinho2015perfect}.  An important result proved at this level of generality is the fact that PST implies that one of the $0-1$ matrices defining the association scheme must be a permutation matrix of order 2 having no fixed points. 
Other examples arise from
Cayley graphs, mostly on abelian groups. For nonabelian groups, examples of PST have been found for some {\it normal} Cayley graphs, that is those in which the connection set is closed under conjugation. In both cases, the graphs belong to the conjugacy association scheme, and the aforementioned result implies that PST requires the existence of a central
involution in the group. Some general results on abelian Cayley graphs are proved in \cite{TanFengCao}. Hypercubes were among the first examples of PST studied
\cite{Christandl} and, more generally, cubelike graphs have been studied in detail in
\cite{BernasconiGS},\cite{CheungGodsil} and \cite{ChanUMCubes}, and numerous
families admitting PST have been found. PST in Integral circulants has been classified
\cite{Basic}  and this result has been generalized to Cayley graphs on abelian groups with cyclic Sylow $2$-subgroups \cite{ArnadottirGodsil2}.
Nonabelian groups with normal Cayley graphs that admit PST include
Dihedral groups of order divisible by 4 \cite{CaoFeng}, and extraspecial $2$-groups \cite{SorciSin}.

\begin{remark}
 As mentioned above,
 if a normal Cayley graph on a group $G$ admits PST, then $G$ must necessarily contain a central involution. In the case $G$ is a group containing a central involution $z$, the Cayley graph on $G$ with $G\setminus \{1,z\}$ as its ``connection'' set, is the complement of the disjoint union of $|G|/2$ copies of $K_{2}$. As we discussed already, this graph admits PST, provided $ |G| \equiv 0 \pmod{4}$. 
 
 While this graph belongs to the conjugacy class scheme on $G$, it is a ``trivially'' occurring example. It is of interest to find examples of groups with central involutions 
 whose conjugacy class schemes contain nontrivial examples of graphs admitting PST.
 
 It is also desirable for these nontrivial examples to be connected, as otherwise they are just disjoint unions of graphs 
 belonging to conjugacy class schemes on a subgroup of $G$.   
\end{remark}

In this paper we construct four families of connected graphs exhibiting PST from the linear groups
$\SL(2,q)$, $\GL(2,q)$ and $\GU(2,q)$, where $q$ is an odd prime power.

Three of these are families of normal Cayley graphs and, as far as we know, are the first such examples in which the underlying groups are not solvable. 

\begin{theorem}\label{thm:cay}
Let $q$ be an odd prime power. If $G$ is one of $\gl(2,q)$, $\gu(2,q)$, or $\sli(2,q)$, there exists a connected normal Cayley graph $\Gamma$ on $G$ admitting PST  such that $\ov{\Gamma}$ is not 
the disjoint union of copies of $K_{2}$.
\end{theorem} 
The connection sets of these Cayley graphs, which have a uniform definition for each family, will be specified later in \S~\ref{sub:glpst}, \S~\ref{sub:gupst}, and \S~\ref{sub:slpst} once the required notation has been established.  At that point the reader may notice a parallelism in the definitions of the $\gl(2,q)$ and $\gu(2,q)$ graphs, as well as in the proofs of these cases of Theorem\ref{thm:cay}, that is congruous with the spirit of "Ennola duality".

The fourth family is constructed in orbital (or Schurian) association schemes, defined in \S\ref{sec:os} below.
(\cite{BannaiIto} and \cite{GMbook} are references for orbital schemes.) When $q$ is a prime power, one of the main results of \cite{gow} shows that the double cosets of $\GL(2,q)$ in $\GL(2,q^2)$ give rise to an orbital association scheme, which we denote by $\gl(2,q^{2})\sslash \gl(2,q)$.

\begin{theorem}\label{thm:orb}
Let $q$ be a prime power with $q\equiv 3 \pmod{4}$.  Then there exists a connected graph
$\Gamma$ belonging to $\gl(2,q^{2})\sslash \gl(2,q)$ admitting PST such that $\ov{\Gamma}$ is not a disjoint union of copies of $K_{2}$.
\end{theorem}

 We believe that these are the first nontrivial examples of perfect state transfer on graphs  belonging to an orbital scheme. The Johnson scheme is also an  orbital scheme but, as mentioned in the Introduction,  the only known graphs belonging to a Johnson scheme admitting PST are the trivial
examples of the Kneser graphs $K(2k,k)$, and their complements  when $\binom{2k}{k}$ is divisible by 4.

To assist the reader we describe next the organization of this paper. There are two sections of preliminaries. \S~\ref{section_assoc} deals with association schemes, recalling the definitions of the conjugacy class schemes and orbital schemes which provide the setting for our later calculations. We also include here the characterization of PST in graphs belonging to an association scheme, which is central to what follows.
In \S~\ref{sec:repandcc} we recall the well-known descriptions 
of the characters and conjugacy classes of the groups
$\GL(2,q)$, $\GU(2,q)$ and $\SL(2,q)$ and establish notation.

The proof of Theorem~\ref{thm:cay} is given in \S~\ref{sec:cayproof}, with one subsection devoted to each
family. These subsections can be read independently. It is here
that the precise definitions of the Cayley graphs are to be found. As mentioned above, the constructions and proofs
for $\GL(2,q)$ and $\GU(2,q)$ are strongly parallel
with $\GU(2,q)$ objects related to their $\GL(2,q)$
counterparts by  "change of signs".

The proof of Theorem~\ref{thm:orb} is carried out in the final section, which requires no material from  \S~\ref{sec:cayproof}.
The constructions of the orbital graphs are given and
the question of PST is reduced to a simple statement (Proposition~\ref{prop:0m4evalue})
which is then established by a series of calculations, requiring
the careful evaluation of character sums over cosets.

\section{PST in association schemes}\label{section_assoc}
A $d$-class association scheme on $n$ vertices, is a set $\mcal{A}=\{A_{0},A_{1},\ldots A_{d}\}$
of   square $0$-$1$ matrices of order $n$ satisfying the following:
\begin{enumerate}
\item $A_{0}=I$;
\item $\sum\limits_{i=0}^{d} A_{i}=J$;
\item $A_{i}^{T} \in \mcal{A}$, for all $0\leq i\leq d$; and
\item $A_{i}A_{j}=A_{j}A_{i} \in \mrm{Span}\left( \mcal{A}\right)$.
\end{enumerate}

The Bose--Mesner algebra $\bb{C}[\mcal{A}]$ of $\mcal{A}$ is the complex algebra generated 
 by the matrices in $\mcal{A}$. By design, elements of $\mcal{A}$ form a basis for $\bb{C}[\mcal{A}]$. It is well known that $\bb{C}[\mcal{A}]$ also has a basis $\mcal{E}=\{E_{0},E_{1},\ldots E_{d}\}$ of primitive idempotents.
 
 Given $i,j \in \{0,1,\ldots, d\}$, let $\lambda_{i,j}$ be such that $A_{i}E_{j}= \lambda_{i,j}E_{j}$. The $\mcal{A} \times \mcal{E}$ array whose $(i,j)$-th entry is $\lambda_{i,j}$, is called the character table of the scheme $\mcal{A}$. Note that the $i$-th row of the character table gives the spectrum of $A_{i}$.

\begin{defi}
A graph $X$ is said to belong to an association scheme $\mcal{A}$ if the adjacency matrix $A_{X}$ of $X$ is in $\bb{C}[\mcal{A}]$. Equivalently, $A_{X}$ is the sum of a subset of the $A_i$, $i\neq0$. 
\end{defi}      
We note that, as we assume graphs to be undirected and without loops, the subset of the $A_i$ in this definition must
be closed under transposition and not include $A_0$.

Coutinho has given necessary and sufficient conditions \cite[Theorems 2.4.4 and 3.2.1, Corollary 3.3.2]{coutinhothesis} for PST between two vertices of  a graph belonging to an association scheme. We now reformulate his results in a  way
which is more convenient for our purposes.

\begin{theorem}\label{PSTAssoc} Let $X$ be a graph  
belonging to a $d$-class association scheme with primitive idempotents
$E_0$, \ldots, $E_d$, and corresponding eigenvalues $\deg(X)=\theta_0 \geq \theta_1 \geq \cdots\geq  \theta_d$.
Let $g=\gcd\{\theta_0-\theta_i\ : \ i=0,\ldots, d\}$. 
Then perfect state transfer occurs at a time $t$ between vertices $a$ and $b$ if and only if all of the following conditions hold.
\begin{itemize}
\item[(a)] The eigenvalues of X are integers.
\item[(b)] The relation of the scheme that contains $(a, b)$, say $T$, is a permutation matrix of order 2 with no fixed points.
\item[(c)] Let $\Phi^+=\{\theta_j : TE_j=E_j\}$ and $\Phi^-=\{\theta_j : TE_j=-E_j\}$. Then 
$v_2(\theta_0-\theta_i)=v_2(g)$ for all $\theta_j\in \Phi^+$ and $v_2(\theta_0-\theta_j)>v_2(g)$
for all $\theta_j\in \Phi^-$. (Here $v_2(m)$ is the $2$-adic valuation
of the integer $m$.)
\end{itemize}
Moreover, $t$ is an odd multiple of $\frac{\pi}{g}$
\end{theorem}
%\qed
\begin{remark}We note that the condition of {\it strong cospectrality} in  \cite[Theorem 2.4.4(i)]{coutinhothesis} is implied by our condition (c) as the latter implies that $\Phi^+\cap\Phi^-=\emptyset$. Hence, if $\tilde E_\theta$
is the idempotent projection for the eigenvalue $\theta$ and 
$\tilde E_\theta=E_{i_1}+\cdots E_{i_t}$ its decomposition into primitive
idempotents then we have $TE_{i_j}=\sigma_\theta E_{i_j}$ for all $j$,
where $\sigma_\theta=\pm1$ is independent of $j$. Thus $\tilde E_\theta e_b= \tilde E_\theta Te_a= \sigma_\theta\tilde E_\theta e_a$.
\end{remark}

In this paper, we are interested in two types of association schemes that arise from finite groups: (a) conjugacy class schemes and (b) orbital (or Schurian) schemes.
 We first describe the conjugacy class scheme.

\subsection{Conjugacy Class Schemes}\label{sec:ccs}
 Let $G$ be a finite group and let $\{C_{0},C_{1},\ldots C_{d}\}$ be the set of conjugacy classes in $G$, with 
 $C_{0}=\{1\}$ Given $0\leq i \leq d$, let $A_{i}$ be the $G \times G$ matrix, with 
\[ A_{i}(g,h) = \begin{cases}
1 & \text{if $hg^{-1} \in C_{i}$}, \\
0 & \text{otherwise.} 
\end{cases},\] for all $g,h \in G$.
It is well-known (see for e.g. \S~3.3 of \cite{GMbook}) that the set $\mcal{A}_{G}=\{A_{i}\ :\ i \in \{0,1, \ldots d\}\}$ forms an association scheme and that its Bose-Mesner algebra is isomorphic to the centre of the group algebra $\bb{C}[G]$, that is, $\bb{C}[\mcal{A}_{G}]=Z\left(\mrm{End}(\bb{C}[G])\right)$.

Let $\mrm{Irr}(G)$ denote the set of irreducible complex characters of $G$. Given $\psi \in \mrm{Irr}(G)$, let $E_{\psi}$ be the $G \times G$ matrix with $E_{\psi}(g,h)= \dfrac{\psi(hg^{-1})\psi(1)}{|G|}$. The set $\{E_{\psi}\ :\ \psi \in \mrm{Irr}(G)\}$ is  the basis of primitive idempotents for $\bb{C}[\mcal{A}_{G}]$. If $c$ is an element in the class $C_{i}$, then we have 
\begin{equation}\label{eq:ccsct}
A_{i}E_{\psi} = \dfrac{|C_{i}|\psi(c^{-1})}{\psi(1)}E_{\psi}.
\end{equation}  

Using Theorem~\ref{PSTAssoc}, if there is a graph admitting PST that belongs to the conjugacy class scheme on $G$, then there is some $t \in G$ such that $A_{t}$ is a permutation matrix of order $2$. We note that $A_{t}$ is a permutation matrix of order $2$ if and only if $t$ is a central involution. In other words, the existence of a graph admitting PST belonging to the conjugacy class scheme on $G$ implies the existence of a central involution in $G$.

Let $G$ be a group possessing a central involution $z$.  Given a graph $\Gamma$ belonging to the conjugacy class scheme on $G$, there exists an inverse-closed set $\mfk{S} \subset G$ which is  a union of conjugacy classes such that $\Gamma:=\mathrm{Cay}(G, \mfk{S})$. Let $A_{\Gamma}$ denote the adjacency matrix of $\Gamma$. Given $\chi \in \irr(G)$, let $\theta_{\chi}$ be the eigenvalue of $A_{\Gamma}$ satisfying $A_{\Gamma}E_{\chi}=\theta_{\chi}E_{\chi}$. Using \eqref{eq:ccsct}, we have 
\[\theta_{\chi}:=  \dfrac{\ov{\chi(\mfk{S})}}{\chi(1)}=\dfrac{{\chi(\mfk{S})}}{\chi(1)}.\] The last equality follows since $\mfk{S}$ is inverse-closed. 

To check if $\Gamma$ admits PST, we apply Theorem~\ref{PSTAssoc}. The role of $T$ in the theorem is played by $A_{\{t\}}$. Using \eqref{eq:ccsct} and Theorem~\ref{PSTAssoc}, we have the following result.
\begin{cor}\label{cor:pstccs}
Let $G$ be a group possessing a central involution $t$, and let  $\mfk{S}$ be an inverse-closed set which is a union of conjugacy classes in $G$. Consider the Cayley graph $\Gamma:=\mathrm{Cay}(G,\mfk{S})$. Given $\chi \in \irr(G)$, define $\theta_{\chi} := \dfrac{{\chi(\mfk{S})}}{\chi(1)}$, \[\Phi^{+}=\{\theta_{\chi}\ :\ \chi(t)/\chi(1)=1\}\] and \[\Phi^{-}=\{\theta_{\chi}\ :\ \chi(t)/\chi(1)=-1\}.\] If,  
\begin{enumerate}[(i)]
    \item $\theta_{\chi} \in \bb{Z}$, for all $\chi \in \irr(G)$; and
    \item there exists $a \in \bb{Z}$ such that $\theta_{\chi} \equiv \begin{cases} a \pmod{4} & \text{if $\chi \in \Phi^{+}$} \\
a+2 \pmod{4} & \text{if $\chi \in \Phi^{-}$,}
\end{cases}$  
\end{enumerate} 
then there is perfect state transfer between vertices $x$ and $xt$ of $\Gamma$, for all $x \in G$. \qed
\end{cor}

\subsection{Orbital Schemes}\label{sec:os}
Let $G$ be a finite group and let $H \leq G$. By $G/H$, we denote the set of left cosets of $H$ in $G$. Let $D$ be a complete set of $(H, H)$--double coset representatives in $G$. Given $d \in D$, we define $A_{d}$ to be the $G/H \times G/H$ matrix satisfying
\[A_{d}(xH, yH) = \begin{cases}
1 & \text{if $x^{-1}y \in HdH$}, \\
0 & \text{otherwise.} 
\end{cases},\]  
for all $xH,yH \in G/H$.

A complex character $\psi$ of $G$ is said to be multiplicity-free if \[\{\left\langle \psi,\chi \right\rangle\ :\ \chi \in \irr(G)\}\subseteq \{0, 1\}.\]  
It is a well-known result (see for e.g., \cite[Corollary~13.5.1]{GMbook}) that $G\sslash H :=\{A_{d}\ : d \in D\}$ is an association scheme if and only if the permutation character $\one_{H}^{G}:=\mrm{Ind}_{H}^{G}\left( \one \right)$ is multiplicity-free. An \emph{orbital scheme} is an association scheme of the form $G\sslash H$, where $(G,H)$ is such that $\one_{H}^{G}$ is multiplicity-free. 

Consider an orbital scheme $G\sslash H$.
Let
\begin{equation}\label{eq:irrsummandsset}
\irr({G\sslash H})  :=\{\chi \in \irr(G)\ :\ \left\langle \psi,\chi \right\rangle=1 \}.
\end{equation}
 Let $S \subset G$ and $\psi$ a complex character of $G$, then by $\chi(S)$, we denote the sum $\sum\limits_{s \in S} \chi(S)$  

Given $\chi \in \irr_{G\sslash H} (G)$, let $\mb{E}_{\chi}$ be the $G/H \times G/H$ matrix satisfying 
\[\mb{E}_{\chi}(xH, yH) = \dfrac{\chi(1)}{|G|}\chi (x^{-1}yH).\] It is a standard result (see for eg. \cite[Corollary~13.8.2]{GMbook}) that $\{\mb{E}_{\chi}\ :\ \chi \in \irr({G\sslash H})\}$ is  the basis of primitive idempotents for $\bb{C}[G\sslash H]$. 

Given $d \in D$ and $\chi \in \irr(G \sslash H)$, we also have (c.f \cite[Corollary~13.9.1]{GMbook})
\begin{equation}\label{eq:osct}
A_{d}\mb{E}_{\chi} = \dfrac{\chi(d^{-1}H)}{|H \cap H^{d}|} \mb{E}_{\chi}.
\end{equation}
By Theorem~\ref{PSTAssoc}, if there is a  graph admitting PST in $G\sslash H$, then there is some $d \in G$ such that $A_{d}$ is a permutation matrix of order $2$. We now give a characterization of orbital schemes containing a permutation matrix of order $2$.
\begin{lemma}
Let $G$ be a finite group and let $H\leq G$ be a subgroup such that $\one_{H}^{G}$ is multiplicity-free. Let $N_{G}(H)$ be the normalizer of $H$ in $G$. Given $d \in D$, the matrix $A_{d} \in G\sslash H$ is a permutation matrix of order $2$ if and only if $d\in N_{G}(H) \setminus H$ and $d^{2} \in H$. 
\end{lemma}
\begin{proof}
Pick a $d \in G$, such that $A_{d}$ is a permutation matrix. Since  $A_{d}$ is a permutation matrix and since $A_{d}(H, dH)=1$, it follows that $A_{d}(H, xH)=0$ for all $xH \neq dH$, 
which means that $HdH=dH$. Hence $d \in N_{G}(H)$.  Conversely, if $d \in N_{G}(H)$, then $HdH=dH$, and thus $A_{d}$ must be a permutation matrix. Thus $A_{d}$ is a permutation matrix if and only if $d \in N_{G}(H)$. 

We observe that given $r,s \in N_{G}(H)$, we have $A_{r}A_{s}=A_{rs}$. We also note that $A_{r}=I$ if and only if $r\in H$. Thus $\{A_{d}\ :\ d \in N_{G}(H)\}$, forms a group isomorphic to $N_{G}(H)/H$. Using this, we conclude that given $d \in G$,  $A_{d}$ is a permutation matrix of order $2$ if and only if $d \in N_{G}(H) \setminus H$ and $d^{2} \in H$.  
\end{proof}

Assume that $G\sslash H$ is such that $N_{G}(H) \gneqq H$ and that there is $t \in N_{G}(H) \setminus H$ such that $t^{2} \in H$. From our discussion above, $A_{t}$ is a permutation matrix of order $2$. Let $D \subset G$ satisfying the following: (i) for all $x,y\in D$ with $x\neq y$, we have $HxH\neq HyH$; (ii) given $x \in D$, there is an $\tilde{x} \in D$ such that $x^{-1} \in H\tilde{x}H$; and (iii) $1 \notin D$. The matrix $A_{D}:=\sum\limits_{d \in D}A_{d}$ must be a symmetric $0$--$1$ matrix. Let $\Gamma_{D}$ be the graph in $G\sslash H$ whose adjacency matrix if $A_{D}$. Given $\chi \in \irr(G\sslash H)$, let $\theta_{\chi}$ be the eigenvalue satisfying $A_{D}\mbf{E}_{\chi}= \theta_{\chi} \mbf{E}_{\chi}$. Using \eqref{eq:osct}, we have 
\[\theta_{\chi}:=  \sum\limits_{d \in D} \dfrac{\chi(d^{-1}H)}{|H \cap H^{d}|}= \sum\limits_{d \in D} \dfrac{\chi(dH)}{|H \cap H^{d}|}.\] 
  Application of  Theorem~\ref{PSTAssoc}  to $\Gamma_{D}$, by setting $T:=A_{t}$, yields the following result.
\begin{cor}\label{cor:ospst}
 Let $G$, $H$, $t$, $D$, and $\Gamma_{D}$ be as described in the previous paragraph. Given $\chi \in \irr(G\sslash H)$, define 
 \[\theta_{\chi}: = \sum\limits_{d \in D} \dfrac{\chi(dH)}{|H \cap H^{d}|},\] \[\Phi^{+}=\left\{\theta_{\chi}\ :\ \dfrac{\chi(tH)}{|H|}=1\right\}\] and \[\Phi^{-}=\left\{\theta_{\chi}\ :\ \dfrac{\chi(tH)}{|H|}=-1\right\}.\] If,  
\begin{enumerate}[(i)]
    \item $\theta_{\chi} \in \bb{Z}$, for all $\chi \in \irr(G\sslash H)$; and
    \item there exists $a \in \bb{Z}$ such that $\theta_{\chi} \equiv \begin{cases} a \pmod{4} & \text{if $\chi \in \Phi^{+}$} \\
a+2 \pmod{4} & \text{if $\chi \in \Phi^{-}$,}
\end{cases}$  
\end{enumerate} 
then there is perfect state transfer between vertices $xH$ and $xtH$ of $\Gamma_{D}$, for all $x \in G$. \qed 
\end{cor}  

\section{Complex representations of $\gl(2,q)$, $\GU(2,q)$ and $\sli(2,q)$.}\label{sec:repandcc}
In this section, we describe the irreducible characters of $\gl(2,q)$ and $\GU(2,q)$ and $\sli(2,q)$.  The facts about $\GU(2,q)$ and $\sli(2,q)$ will be needed only for
our study of the Cayley graphs on the respective groups,
while those for $\gl(2,q)$ will be needed for both the Cayley graphs and the orbital scheme graphs.
 If $\F$ is a field, we shall denote by $\widehat{\F^\times}$ the group of multiplicative characters $\F^\times$ with complex values.

\subsection{Characters of $\gl(2,q)$}\label{sec:repandccgl}
Let $q$ be an odd prime power.
In this section, we describe conjugacy classes and complex irreducible representations of the finite general linear group $\gl(2,q)$. 
The irreducible complex representations of this group were first classified by Schur \cite{Schur1907} and Jordan \cite{jordan1907group}. For a modern exposition, we refer the reader to \cite[\S~5.2]{fulton2013representation}.

We first start by describing the conjugacy classes of $\gl(2,q)$.  
 Let $q$ be an odd prime power and let $\F_{q}$ and $\F_{q^{2}}$ denote the finite fields of sizes $q$ and $q^{2}$ respectively. 
Fix $\Delta$ to be a fixed non-square in $\F_{q}$, and let $\sqrt{\Delta}\in\F_{q^2}$ denote a fixed root of $t^{2}-\Delta \in \F_{q}[t]$.  The set $\{1,\sqrt{\Delta}\}$ is a basis for $\F_{q^{2}}$ over $\F_{q}$. Given $z \in \F_{q^{2}}$, let $x_{z},y_{z} \in \F_{q}$ be such that $z=x_{z} + y_{z} \sqrt{\Delta}$.

Given $x,y \in \F_{q}^{\times}$ with $x\neq y$, define $c_{1}(x) :=\begin{bmatrix}
x & 0 \\
0 & x
\end{bmatrix} $, $c_{2}(x) :=\begin{bmatrix}
x & 1 \\
0 & x
\end{bmatrix} $, $c_{3}(x,y) := \begin{bmatrix}
x & 0 \\
0 & y
\end{bmatrix} $. Given $z \in \F_{q^{2}} \setminus \F_{q}$, define $c_{4}(z):= \begin{bmatrix}
x_{z} & \Delta y_{z} \\
y_{z} & x_{z}
\end{bmatrix} $. A central element in $G$ is necessarily one of $\{c_{1}(x)\ :\ x \in \F_{q}^{\times}\}$. A non-central element with exactly one eigenvalue $x \in \F_{q}^{\times}$ is conjugate to $c_{2}(x)$. A non-central element with two distinct eigenvalues $x,y \in \F_{q}^{\times}$ is conjugate to $c_{3}(x,y)$. Finally, an element with no eigenvalues in $\F_{q}^{\times}$ is conjugate to one of $\{c_{4}(z)\ :\ z \in \F_{q^{2}} \setminus \F_{q}\}$.

Irreducible characters of $\gl(2,q)$ can be split into four categories, by degree. We first describe the linear characters.

Let $\mrm{det}: \gl(2,q) \to \F_{q}^{\times}$ denote the determinant map. Every linear character of $\gl(2,q)$ is of the form $\lambda\circ\mrm{det}$,
where $\lambda$ is a multiplicative character
of $\F_{q}^\times$. For ease of writing, we shall sometimes abuse notation to identify $\lambda\circ\mrm{det}\in \mrm{Irr}(G)$ with $\lambda\in\widehat{\F_{q}^\times}$.

Next, we describe the so called parabolically induced characters.  Let $\bo(q)$ denote the subgroup of upper triangular matrices in $\gl(2,q)$. Given $\chi_{1},\chi_{2} \in \widehat{\F_{q}^{\times}}$, we get a character $\chi:=[\chi_{1},\chi_{2}]$ of 
$\F_{q}^{\times}\times \F_{q}^{\times}$ and, by inflation,
a character of $\bo(q)$, which we denote by the same symbol, given  by 
\begin{equation}\label{eq:linchpar}
[\chi_1, \chi_2]\left( \begin{bmatrix}
x & z \\
0 & y
\end{bmatrix} \right) =\chi_{1}(x)\chi_{2}(y).
\end{equation} We obtain a character $I[\chi_{1},\chi_{2}]:= \mrm{Ind}_{\bo(q)}^{\gl(2,q)}(\chi)$ of $\gl(2,q)$, via induction. It is well known that, provided $\chi_{1} \neq \chi_{2}$, the character $I[\chi_{1},\chi_{2}]$ is an irreducible character of degree $q+1$. Such characters of degree $q+1$ will be referred to as principal series characters.
It is also known that
 for  $\chi \in \widehat{\F_{q}^{\times}}$  the associated linear character $\chi:=\chi\circ\mrm{det}$ of $G$ is an irreducible summand of $I[\chi,\chi]$ and that 
 $S_{\chi}:=I[\chi,\chi]-\chi$ is an irreducible character of degree $q$. We shall refer to such characters as Steinberg type characters.
 
A character $\pi \in \irr(\gl(2,q))$ is said to be cuspidal if $\left\langle \pi, I[\chi] \right\rangle =0$, for all $\chi \in \irr\left( \F_{q}^{\times} \times \F_{q}^{\times} \right)$. There are $q(q-1)/2$ cuspidal characters and each one of them has degree $q-1$. Cuspidal characters are indexed by $\{\mu \in \mrm{Irr}(\F_{q^{2}}^{\times})\ :\ \mu \neq \mu^{q}\}$.

The character table of $\gl(2,q)$ is given as Table~\ref{tab:gl2qct}.
\begin{table}[ht]
\centering
\begin{tabular}{|c||c || c | c | c | c|}
\hline
& & $c_{1}(x)$ & $c_{2}(x)$ & $c_{3}(x,y)$ & $c_{4}(z)$ \\
\hline
& Size & $1$ & $q^{2}-1$ & $q(q+1)$ & $q(q-1)$ \\
\hline\hline
Character & Type & & & &
\\
\hline
$\lambda$ & Linear & $\lambda(x^{2})$ & $\lambda(x^{2})$ & $\lambda(xy)$ & $\lambda(z^{q+1})$ \\
\hline
$S_{\lambda}$  & Steinberg & $q\lambda(x^{2})$ & $0$ & $\lambda(xy)$ & $-\lambda(z^{q+1})$ \\
\hline
$\pi[\mu]$ & Cuspidal & $(q-1)\mu(x)$ & $-\mu(x)$ & $0$ & $-\mu(z)-\mu(z^{q})$ \\
\hline
$I[\chi_{1},\chi_{2}]$ & Principal series & $(q+1)\chi_{1}(x)\chi_{2}(x)$ & $\chi_{1}(x)\chi_{2}(x)$ & $\chi_{1}(x)\chi_{2}(y) + \chi_{1}(y)\chi_{2}(x)$ & $0$ \\
\hline
\end{tabular}
\caption{Character table of $\gl(2,q)$.}
\label{tab:gl2qct}
\end{table}
\subsection{Characters of $\GU(2,q)$}\label{sec:gu2c}
Let $q$ be an odd prime power. Fix a hermitian form on $\bb{F}_{q^{2}}^{2}$. By $\GU(2,q)$, we mean the subgroup of form-preserving matrices in $\gl(2,q^2)$. In this section, we describe the complex irreducible characters of $\GU(2,q)$. We start by defining $E$ to be the unique subgroup of $\bb{F}_{q^{2}}^{\times}$, of index $q-1$.

The conjugacy classes and irreducible characters of $\GU(2,q)$ were characterized by Ennola \cite{ennolaconjugacy, Ennola}. Ennola proved that $g,h \in GU(2,q)$ are conjugate if and only if they are similar as matrices, and he also determined the rational canonical forms of elements in $\GU(2,q)$. In other words, a $\GU(2,q)$ conjugacy class is uniquely determined by the rational cannonical form of its elements. The information about various conjugacy classes and their rational canonical form is given in Table~\ref{tab:gucc}.

\begin{table}[h]
    \centering
    \begin{tabular}{|c|c|c|c|}
    \hline
         Class & Rational form  & Size & Number of classes of this type   \\
          \hline
          $C_{1}(x)$, $x\in E$  & $\begin{bmatrix}
              x &0 \\
              0 & x
          \end{bmatrix} $ & $1$  & $q+1$\\
          \hline
          $C_{2}(x)$, $x\in E$  & $\begin{bmatrix}
              x & 1 \\
              0 & x
          \end{bmatrix} $ & $q^{2}-1$ & $q+1$ \\ 
          \hline
          $C_{3}(x,y)$, $x, y\in E$, $x\neq y$  & $\begin{bmatrix}
              x & 0 \\
              0 & y
          \end{bmatrix} $ & $q(q-1)$ & $\dfrac{q(q+1)}{2}$ \\
          \hline
          $C_{4}(z)$, $z\in \bb{F}_{q^{2}}^{\times} \setminus E$  & $\begin{bmatrix}
              z & 0 \\
              0 & z^{-q}
          \end{bmatrix} $ & $q(q+1)$ & $\dfrac{q^{2}-q-2}{2}$ \\
          \hline
    \end{tabular}
    \caption{Conjugacy Classes of $\GU(2,q)$}
    \label{tab:gucc}
\end{table}

Characters of $\GU(2,q)$ were also first described by Ennola \cite{Ennola}. The construction of characters is very similar to those of $\gl(2,q)$. We now describe the irreducible characters. 

Every linear character of $\GU(2,q)$ is of the form  $\lambda \circ det$, where $\lambda$ be a linear character of $E$. For ease of notation, we use $\lambda$ to denote $\lambda \circ det$. For every linear character $\lambda$, there is an irreducible character $S_{\lambda}$ of degree $q$. Every irreducible character of degree $q$ is of this form.
Given distinct linear characters $\lambda_{1},\lambda_{2}$ of $E$, there is a character $I[\lambda_{1},\lambda_{2}] \in \irr(\GU(2,q))$, of degree $q-1$. Every irreducible character of degree $q-1$ is of this form. Given a linear character $\mu$ of $\bb{F}_{q^{2}}^{\times}$ with $\mu \neq \mu^{-q}$ (that is, $\mu|_{\bb{F}_{q}^{\times}}$ is non-trivial), there is a character $\pi[\mu]$ of degree $q+1$. Every irreducible character of degree $q+1$ is of this form. Any character in $\irr(G)$ must be one of the irreducible characters described in this paragraph. The character table of $\GU(2, q)$ is given in Table~\ref{tab:guct}.

\begin{table}[h]
    \centering
    \begin{tabular}{|c||c|c|c|c|}
    \hline
         & $C_{1}(x)$ & $C_{2}(x)$ & $C_{3}(x,y)$ & $C_{4}(z)$ \\
         \hline
       $\lambda$  & $\lambda(x^{2})$ & $\lambda(x^{2})$ & $\lambda(xy)$ & $\lambda(z^{1-q})$ \\
         \hline
       $S_{\lambda}$  & $q\lambda(x^{2})$ & $0$ & -$\lambda(xy)$ & $\lambda(z^{1-q})$ \\
       \hline
       $I[\lambda_{1},\lambda_{2}]$ & $(q-1)\lambda_{1}(x)\lambda_{2}(x)$ & $-\lambda_{1}(x)\lambda_{2}(x)$ & $-(\lambda_{1}(x)\lambda_{2}(y)+\lambda_{2}(x)\lambda_{1}(y))$ & $0$ \\
       \hline
       $\pi[\mu]$ & $(q+1)\mu(x)$ & $\mu(x)$ & $0$ & $\mu(z)+\mu(z^{-q})$\\
       \hline
    \end{tabular}
    \caption{$\GU(2,q)$ Character Table}
    \label{tab:guct}
\end{table}
 
\subsection{Characters of $\sli(2,q)$.}\label{sec:sl2cc}
Let $q$ be an odd prime power. In this section, we describe the character table of $\sli(2,q)$.
The irreducible complex representations of $\sli(2,q)$ were first characterized by Schur \cite{Schur1907} and Jordan \cite{jordan1907group}. We can obtain irreducible representations of $\sli(2,q)$ as irreducible summands of restrictions of irreducible representations of $\gl(2,q)$. In fact, all complex irreducible representations of $\sli(2,q)$ can be obtained this way. More details on this approach can be found in many standard texts such as \cite{fulton2013representation}. 

The conjugacy classes of $\sli(2,q)$ are well known (see for instance  \cite[\S~5.2]{fulton2013representation}). Let $\Delta$ be a fixed non-square in $\bb{F}_{q^{\times}}$. Conjugacy classes of $\sli(2,q)$ are described in Table~\ref{tab:sl2cc}.

\begin{table}[h!]
    \centering
    \begin{tabular}{|c|c|}
    \hline
       Representative  &  Size of the class   \\
       \hline
        $I$ & $1$ \\ 
        \hline
        $-I$ & $1$ \\
 \hline
        $d_{2}(1,1):= \begin{bmatrix}
            1 & 1 \\
            0 & 1
        \end{bmatrix}$  & $\dfrac{q^{2}-1}{2}$ \\
 \hline
        $d_{2}(1,\Delta):= \begin{bmatrix}
            1 & \Delta \\
            0 & 1
        \end{bmatrix}$  & $\dfrac{q^{2}-1}{2}$ \\
        \hline
        $d_{2}(-1,1):= \begin{bmatrix}
            -1 & 1 \\
            0 & -1
        \end{bmatrix}$  & $\dfrac{q^{2}-1}{2}$ \\
        \hline
        $d_{2}(-1,\Delta):= \begin{bmatrix}
        -1 & \Delta \\
            0 & -1
        \end{bmatrix}$  & $\dfrac{q^{2}-1}{2}$ \\
        \hline 
        $d_{3}(x):= \begin{bmatrix}
            x & 0 \\
            0 & x^{-1}
        \end{bmatrix}$, $x \in \bb{F}_{q}^{\times}\setminus \{\pm 1\}$ & $q(q+1)$ \\
        \hline 
        $d_{4}(x, y):= \begin{bmatrix}
            x & \Delta y \\
            y & x
        \end{bmatrix}$, $(x,y) \in \bb{F}_{q} \times \bb{F}_{q}^{\times}$ and $x^{2}-y^{2}\Delta=1$ & $q(q-1)$\\
        \hline
    \end{tabular}
    \caption{Conjugacy classes of $\sli(2,q)$}
    \label{tab:sl2cc}   
\end{table}

We now describe the irreducible characters of $\sli(2,q)$. Let $E$ denote the unique subgroup of index $q-1$ in $\bb{F}_{q^{2}}^{\times}$. Let $\mu_{s}$ denote a linear character of $E$, of order $2$. We extend $\mu_{s}$ to a character of $\bb{F}_{q^{2}}^{\times}$ and denote this extension by $\mu_{s}$ as well. Consider the restriction $U:=\pi[\mu_{s}]|_{\sli(2,q)}$ of the cuspidal character $\pi[\mu_{s}]$. Elementary computations show that $\left\langle U, U \right\rangle =2$. It is known (c.f \cite[\S~5.2]{fulton2013representation}) that $U$ is a direct sum of two distinct irreducible characters of the same degree, $\dfrac{q-1}{2}$. We denote the irreducible summands of $U$ by $U^{+}$ and $U^{-}$. We just described two irreducible characters of degree $\dfrac{q-1}{2}$.

Let $\zeta$ denote the quadratic character of $\bb{F}_{q}^{\times}$ and consider the restriction $Z:=I[\zeta,1]|_{\sli(2,q)}$ of the irreducible character $I[\zeta, 1]$ of $\gl(2,q)$. Just as above, $Z$ is a sum of two distinct irreducible characters of the same degree, $\dfrac{q+1}{2}$. We denote the irreducible summands of $Z$ by $Z^{+}$ and $Z^{-}$. We just described two irreducible characters of degree $\dfrac{q+1}{2}$.

If $\mu$ is a linear character of $\bb{F}_{q^{2}}^{\times}$, with $\mu \neq \mu^{q}$ and $\mu|_{E} \neq \mu_{s}|_{E}$, then $\pi[\mu]|_{\sli(2,q)}$ is an irreducible character. We also note that given another is a linear character $\rho$ of $\bb{F}_{q^{2}}^{\times}$, with $\rho \neq \rho^{q}$, then $\pi[\mu]|_{\sli(2,q)}=\pi[\rho]|_{\sli(2,q)}$ if and only if $\mu|_{E} \in \{\rho|_{E},\rho^{q}|_{E}\}$. Given a linear character $\mu$ of $E$, with $\mu \neq \mu^{-1}$, we denote $P[\mu]$, to be the restriction of $\pi[\tilde{\mu}] \in \irr(\gl(2,q))$ to $\sli(2,q)$, where $\tilde{\mu}$ is any extenstion of $\mu$ to $\bb{F}_{q^{2}}^{\times}$. We just described $\dfrac{q-1}{2}$ irreducible characters of degree $(q-1)$.   

Given a distinct linear character $\lambda$ of $\bb{F}_{q}^{\times}$, with $\lambda^2 \neq 1$ the character $I[\lambda, 1]$, restricts to an irreducible character of $\sli(2,q)$. We define $I[\lambda] :=I[\lambda, 1]|_{\sli(2,q)}$. These $I[\lambda]$'s form a set of $\dfrac{q-3}{2}$ irreducibles of degree $q+1$.

Consider the $2$-transitive action of $\sli(2,q)$ on the $1$-dimensional susbspaces of $\bb{F}_{q}^{2}$. Let $\sigma$ be the degree $q$  irreducible character, such that $\sigma+1$ is the permutation character associated with this action. We note that $\sigma$ is the restriction of any Stienberg type character of $\gl(2,q)$. We have so far described $q+3$ distinct irreducible characters of $\sli(2,q)$. Including the trivial character, we have $q+4$ irreducible characters, which is the same as the number of conjugacy classes, and therefore, these are all the irreducible characters. For our purposes, we do not require the entire character table, but only require the character values on $\pm I$, $d_{2}(\pm 1, 1)$ and $d_{2}(\pm 1, \Delta)$. Character values of any character in $\irr(\sli(2,q)) \setminus \{U^{\pm}Z^{\pm}\}$ can be read off the character table \ref{tab:gl2qct} of $\gl(2,q)$. The character values of $U^{\pm}$, $Z^{\pm}$ take a bit more work, but are well-known. The values taken by characters in $\{U^{\pm},Z^{\pm}\}$ on elements of $ \{d_{2}(\pm 1, 1),d_{2}(\pm 1, \Delta)\}$ are given in Table~\ref{tab:sl}.

\begin{table}[ht]
    \centering
    \begin{tabular}{|c|c|c|c|c|}
    \hline
         & $d_{2}(1,1)$ & $d_{2}(1,\Delta)$ & $d_{2}(-1,1)$ & $d_{2}(-1,\Delta)$ \\
         \hline
         $U^{\pm}$ & $\frac{-1 \pm \sqrt{\zeta(-1)q}}{2}$ & $\frac{-1 \mp \sqrt{\zeta(-1)q}}{2}$  & $\frac{ \zeta(-1) \mp \sqrt{\zeta(-1)q}}{2}$ & $\frac{ \zeta(-1) \pm \sqrt{\zeta(-1)q}}{2}$   \\
         \hline
         $Z^{\pm}$ & $\frac{1 \pm \sqrt{\zeta(-1)q}}{2}$ & $\frac{1 \mp \sqrt{\zeta(-1)q}}{2}$  & $\frac{ \zeta(-1) \pm \sqrt{\zeta(-1)q}}{2}$ & $\frac{ \zeta(-1) \mp \sqrt{\zeta(-1)q}}{2}$   \\ 
         \hline
    \end{tabular}
    \caption{Partial Character Table of $\sli(2,q)$}
    \label{tab:sl}
\end{table}

\section{Proof of Theorem~\ref{thm:cay}}\label{sec:cayproof}
Theorem~\ref{thm:cay} considers graphs 
belonging to the
conjugacy class scheme. We prove this using Corollary~\ref{cor:pstccs}. 
\subsection{PST in $\gl(2,q)$}\label{sub:glpst}
Let $q$ be an odd prime power and let $G=\gl(2,q)$. Then $G$ has a unique central involution, namely, $t := \begin{bmatrix}
    -1 & 0 \\
    0 & -1
\end{bmatrix}$. 
Let $S$ denote the set of squares in $\bb{F}_{q}^{\times}$ and $N:=\bb{F}_{q}^{\times} \setminus S$. 

We consider the set $\mfk{S}$ of elements which are conjugate to an element of the set
\[\{c_{3}(1,-1)\} \bigcup \left(\bigcup\limits_{x \in \bb{F}_{q}^{\times}}c_{2}(x) \right) \bigcup \left( \bigcup\limits_{\substack{z \in \bb{F}_{q^{2}}^\times\setminus \bb{F}_{q}^\times \\ z^{q+1}\in \{1\} \cup N}} c_{4}(z)\right).\]
Here, $c_{2}(x)$, $c_{3}(x,y)$, and $c_{4}(z)$ are as defined in \S~\ref{sec:repandccgl}. 
We will prove that the Cayley graph $\Gamma := \mathrm{Cay}\left(G,\mfk{S} \right)$ --- a graph belonging to  the conjugacy class scheme over $G$ --- possesses PST. This shall be done by proving that $\Gamma$ satisfies the premise of Corollary~\ref{cor:pstccs}. We do this in a series of lemmas.

\begin{lemma}\label{lem:gl2inte}
The eigenvalues of $\Gamma$ are integral.
\end{lemma}
\begin{proof}
  It is well known that, given a conjugacy class $C$ and an irreducible character $\chi$, $\dfrac{\chi(C)}{\chi(1)}$ is an algebraic integer. As $\mfk{S}$ is a union of rational conjugacy classes, it follows that, for any $\chi \in \irr(G)$, $\theta_{\chi}= {\chi(S)}/ \chi(1)$ is an integer.    
\end{proof}

Our next step will be to show that all eigenvalues in $\Phi^{+}$ are congruent to $2\pmod{4}$ and the ones in $\Phi^{-}$ are $0\pmod{4}$. We first prove some auxiliary lemmas. The following result is an immediate consequence of orthogonality of characters.

\begin{lemma}\label{lem:square}
Let $\lambda$ be a linear character of $\bb{F}_{q}^{\times}$.
\[\sum\limits_{x \in \bb{F}_{q}^{\times}} \lambda(x^{2})= 2\lambda(S) = \begin{cases}
    q-1 & \text{if $S \subset \mrm{Ker}(\lambda)$} \\
    0 & \text{otherwise.}
\end{cases} \]    \qed
\end{lemma}
 
 The norm map $\mrm{Nm} : \bb{F}_{q^{2}}^{\times} \to \bb{F}_{q}^{\times}$ is a surjective homomorphsim satisfying $N(z)=z^{q+1}$, for all $z \in \bb{F}_{q^{2}}^{\times}$. The kernel $E$ of this map is the unique subgroup of $\bb{F}_{q^{2}}^{\times}$, of order $q+1$. Given $x \in \bb{F}_{q}^{\times}$, the set $E_{x}:=\mrm{Nm}^{-1}(x)$ is a coset of $E$. For any $z \in \bb{F}_{q}$, we note that $z^{q+1}=z^{2}$. Therefore, $E_{x} \cap \bb{F}_{q}^{\times} \neq \emptyset$ if and only if $x $ is a square in $\bb{F}_{q}^{\times}$.  Moreover, we have
 \begin{subequations}\label{eq:Ecoset}
 \begin{align}
     |E_{x} \cap \bb{F}_{q}^{\times}| &= \begin{cases} 2 & \text{if $x\in S$} \\
     0 & \text{otherwise.}
     \end{cases} \\
     E \cap \bb{F}_{q}^{\times} & = \{1,-1\}.
  \end{align}   
 \end{subequations}
Therefore, we have 

\begin{equation}\label{eq:set4}
    \{z \in \bb{F}_{q^{2}}^\times\setminus \bb{F}_{q}^\times \ : \ z^{q+1}\in \{1\} \cup N\} =   \left(\bigcup\limits_{x \in N} E_{x} \right) \cup (E\setminus \{ 1, -1\}). 
\end{equation}

We now prove one last technical result before going back to our main result.

\begin{lemma}\label{lem:gl2csums}
    Let $\lambda$ be a non-trivial linear character of $\bb{F}_{q}^{\times}$ and let $\mu$ be linear character of $\bb{F}_{q^{2}}^{\times}$ with $\mu \neq \mu^{q}$. 
    Then, we have 
    \begin{enumerate}
        \item $\sum\limits_{\substack{z \in \bb{F}_{q^{2}}^\times\setminus \bb{F}_{q}^\times \\ z^{q+1}\in \{1\} \cup N}} \lambda(z^{q+1})= (q-1) -(q+1)\lambda(S)$ 
        \item $\sum\limits_{\substack{z \in \bb{F}_{q^{2}}^\times\setminus \bb{F}_{q}^\times \\ z^{q+1}\in \{1\} \cup N}}\mu(z) = -(1+ \mu(-1))$
    \end{enumerate}   
\end{lemma}
 \begin{proof}
     Given $x \in \bb{F}_{q}^{\times}$, we recall that $E_{x}=\{z \in \bb{F}_{q^{2}}^{\times}\ :\ z^{q+1}=x\}$ is a coset of $E$, and thus a set of size $q+1$. Let $\lambda$ be a non-trivial linear character of $\bb{F}_{q}^{\times}$.
     Using \eqref{eq:Ecoset} and \eqref{eq:set4}, we have  
     \begin{align*}
          \sum\limits_{\substack{z \in \bb{F}_{q^{2}}^\times\setminus \bb{F}_{q}^\times \\ z^{q+1}\in \{1\} \cup N}} \lambda(z^{q+1})&= |E\setminus  \{1, -1\}| \lambda(1) + \sum\limits_{x \in N} |E_{x}\setminus \bb{F}_{q}^{\times}| \lambda(x) \\
          &= (q-1)\lambda(1) +(q+1) \lambda(N).
     \end{align*}
     
     Since $\lambda$ is non-trivial, by orthogonality of characters, we have $0=\lambda(\bb{F}_{q}^{\times})=\lambda(N)+\lambda(S)$, and thus $\lambda(N)=-\lambda(S)$. This proves the first equality.

Let $\mu$ be a non-trivial linear character of $\bb{F}_{q^{2}}^{\times}$, with $\mu \neq \mu^{q}$. Since $E$ is a subgroup of index $q-1$ in the cyclic group $\bb{F}_{q^{2}}^{\times}$, we see that $E :=\{z^{q-1}\ : z \in \bb{F}_{q^{2}}^{\times}\}$. Since $\mu \neq \mu^{q}$, there exists some $z \in \bb{F}_{q^{2}}^{\times}$ with $\mu(z^{q-1})=\mu^{q-1}(z) \neq 1$. Therefore, $\mu|_{E}$ is a non-trivial character and thus $\mu(E)=0$. For every $x \in \bb{F}_{q}^{\times}$, $E_{x}$ is a coset of $E$, and thus $\mu(E_{x})=0$. Using \eqref{eq:set4}, we have
\begin{align*}
    \sum\limits_{\substack{z \in \bb{F}_{q^{2}}^\times\setminus \bb{F}_{q}^\times \\ z^{q+1}\in \{1\} \cup N}} \mu(z) &= \mu(E)-\mu(\{1,-1\}) +\sum\limits_{x\in N} \mu(E_{x}) \\
   &= -(1+\mu(-1)).
\end{align*}
 \end{proof}

 We now show that $\Gamma$ possesses PST. We start by computing the eigenvalues of $\Gamma$. We begin with the eigenvalues indexed by linear characters of $G$.
 
 \begin{lemma}\label{lem:lin}
    If $\lambda$ is a linear character of $G$, then \begin{equation*}
    \theta_{\lambda} = \begin{cases} q(q+1)\lambda(-1) +q\dfrac{(q-1)^{2}}{2}  & \text{if $S\not\subseteq \mrm{Ker}(\lambda)$,} \\
      q(q+1)\lambda(-1) + (q^{2}-1)(q-1) +q\dfrac{(q-1)^{2}}{2} -q\dfrac{(q+1)(q-1)^{2}}{4} & \text{if $\lambda$ is the quadratic character,} \\
      q(q+1) + (q^{2}-1)(q-1) +q \dfrac{(q-1)^{2}}{2} + q\dfrac{(q+1)(q-1)^{2}}{4} & \text{if $\lambda$ is trivial.}
    \end{cases}
\end{equation*}
Moreover, $\theta_{\lambda}\equiv 2 \pmod{4}$ and $\theta_{\lambda} \in \Phi^{+}$.
\end{lemma}
\begin{proof}
    Let $\lambda$ be a complex linear character of $G$. Recalling our notation from \S~\ref{sec:repandccgl}, we can also treat $\lambda$ as a complex linear character of $\bb{F}_{q}^{\times}$. We recall that there are exactly $q(q+1)$ elements conjugate to $c_{3}(1,-1)$; exactly $q^{2}-1$ elements conjugate to each $c_{2}(x)$; and exactly $q(q-1)$ elements conjugate to each $c_{4}(z)$. We also recall that $c_{2}(x)=c_{2}(y)$ if and only if $x=y$; and that $c_{4}(z)=c_{4}(w)$ if and only if $\{z,\ov{z}\} = \{w,\ov{w}\}$.  
We now have  
\begin{align*}
\theta_{\lambda} &= {\lambda(\mfk{S})}/\lambda(1) \notag \\
& = q(q+1)\lambda(c_{3}(1,-1))+(q^{2}-1) \sum\limits_{x \in \bb{F}_{q}^{\times}} \lambda(c_{2}(x))+ \dfrac{q(q-1)}{2}\sum \limits_{\substack{z \in \bb{F}_{q^{2}}^\times\setminus \bb{F}_{q}^\times \\ z^{q+1}\in \{1\} \cup N}} \lambda(c_{4}(z)) \notag\\
&=q(q+1)\lambda(-1)+(q^{2}-1) \sum\limits_{x \in \bb{F}_{q}^{\times}} \lambda(x^{2})+ \dfrac{q(q-1)}{2}\sum \limits_{\substack{z \in \bb{F}_{q^{2}}^\times\setminus \bb{F}_{q}^\times \\ z^{q+1}\in \{1\} \cup N}} \lambda(z^{q+1}). 
\end{align*}
Using Lemma~\ref{lem:square} and Lemma~\ref{lem:gl2csums}, it follows that 

\begin{equation*}
    \theta_{\lambda} = \begin{cases} q(q+1)\lambda(-1) +q\dfrac{(q-1)^{2}}{2}  & \text{if $S\not\subseteq \mrm{Ker}(\lambda)$,} \\
      q(q+1)\lambda(-1) + (q^{2}-1)(q-1) +q\dfrac{(q-1)^{2}}{2} -q\dfrac{(q+1)(q-1)^{2}}{4} & \text{if $\lambda$ is the quadratic character,} \\
      q(q+1) + (q^{2}-1)(q-1) +q \dfrac{(q-1)^{2}}{2} + q\dfrac{(q+1)(q-1)^{2}}{4} & \text{if $\lambda$ is trivial.}
    \end{cases}
\end{equation*}

As $q^{2}-1 \equiv 0\pmod{8}$, it follows that $\theta_{\lambda} \equiv q(q+1)\lambda(-1)+\dfrac{q(q-1)^{2}}{2} \pmod{4}$.
Using $\lambda(-1)=\pm 1$ and $q\equiv \pm 1 \pmod{4}$, we conclude that  $\theta_{\lambda} \equiv 2 \pmod{4}$.
Finally, we observe that $\lambda(t)=\lambda(1)=1$, and thus $\theta_{\lambda} \in \Phi^{+}$.
\end{proof}

\begin{lemma}\label{lem:st}
    If $\lambda$ is a linear character of $G$, then \begin{equation*}
    \theta_{S_{\lambda}} = \begin{cases} (q+1)\lambda(-1) +\dfrac{(q-1)^{2}}{2}  & \text{if $S\not\subseteq \mrm{Ker}(\lambda)$,} \\
      (q+1)\lambda(-1)  + \dfrac{(q-1)^{2}}{2} -\dfrac{(q+1)(q-1)^{2}}{4}& \text{if $\lambda$ is the quadratic character,} \\
      (q+1) + \dfrac{(q-1)^{2}}{2} +\dfrac{(q+1)(q-1)^{2}}{4} & \text{if $\lambda$ is trivial.}
    \end{cases}
\end{equation*} Moreover, $\theta_{S_{\lambda}}\equiv 2 \pmod{4}$ and $\theta_{S_{\lambda}} \in \Phi^{+}$.
\end{lemma}
\begin{proof}
    Given a linear character $\lambda$ of $ \bb{F}_{q}^{\times}$, consider the character $S_{\lambda}$, and the eigenvalue $\theta_{S_{\lambda}}$. We have 
\begin{align*}
\theta_{S_{\lambda}} &= {S_{\lambda}(\mfk{S})}/S_{\lambda}(1) \notag \\
& = (q+1)\lambda(c_{3}(1,-1))+\dfrac{(q^{2}-1)}{q} \sum\limits_{x \in \bb{F}_{q}^{\times}} \lambda(c_{2}(x))+ \dfrac{q(q-1)}{2q}\sum \limits_{\substack{z \in \bb{F}_{q^{2}}^\times\setminus \bb{F}_{q}^\times \\ z^{q+1}\in \{1\} \cup N}} \lambda(c_{4}(z)) \notag\\
&=(q+1)\lambda(-1)+0- \dfrac{(q-1)}{2}\sum\limits_{\substack{z \in \bb{F}_{q^{2}}^\times\setminus \bb{F}_{q}^\times \\ z^{q+1}\in \{1\} \cup N}} \lambda(z^{q+1}). 
\end{align*}
Using Lemma~\ref{lem:square} and Lemma~\ref{lem:gl2csums}, it follows that 

\begin{equation*}
    \theta_{S_{\lambda}} = \begin{cases} (q+1)\lambda(-1) +\dfrac{(q-1)^{2}}{2}  & \text{if $S\not\subseteq \mrm{Ker}(\lambda)$,} \\
      (q+1)\lambda(-1)  + \dfrac{(q-1)^{2}}{2} -\dfrac{(q+1)(q-1)^{2}}{4}& \text{if $\lambda$ is the quadratic character,} \\
      (q+1) + \dfrac{(q-1)^{2}}{2} +\dfrac{(q+1)(q-1)^{2}}{4} & \text{if $\lambda$ is trivial.}
    \end{cases}
\end{equation*}
 Using $\lambda(-1)=\pm 1$ and $q\equiv \pm 1 \pmod{4}$, we conclude that  $\theta_{S_{\lambda}} \equiv 2 \pmod{4}$.
Finally, we observe that $S_{\lambda}(t)/S_{\lambda}(1)=1$, and thus $\theta_{S_{\lambda}} \in \Phi^{+}$.
\end{proof}

Now, we consider the eigenvalues indexed by cuspidal irreducible characters of $G$. 

\begin{lemma}\label{lem:cusp}
Let $\mu$ be a linear character of $\bb{F}_{q^{2}}^{\times}$ with $\mu \neq \mu^{q}$. Then,
\[ \theta_{\pi[\mu]} = \begin{cases} 0  & \text{if $\mu(-1)=-1$,} \\
 2q  & \text{if $\mu(-1)=1$ and $\bb{F}_{q}^{\times} \not\subseteq \mrm{Ker}(\mu)$,}\\
-(q^{2}-1) +2q  & \text{if $\bb{F}_{q}^{\times}\subseteq \mrm{Ker}(\mu)$.}
\end{cases}\]
Moreover, 
\[\theta_{\pi[\mu]} \equiv \begin{cases} 
0 \pmod{4} & \text{if $\theta_{\pi[\mu]} \in \Phi^{-}$,} \\
2 \pmod{4} & \text{if $\theta_{\pi[\mu]} \in \Phi^{+}$.} \\
\end{cases} \]
\end{lemma}
\begin{proof}
Let $\mu$ be a linear character of $\bb{F}_{q^{2}}^{\times}$, with $\mu \neq \mu^{q}$. Consider the cuspidal character $\pi[\mu]$. We have 
\begin{align*}
 \theta_{\pi[\mu]} &= {\pi[\mu](\mfk{S})}/{\pi[\mu](1)} \notag \\
 & = \dfrac{(q^{2}-1)}{q-1}\sum\limits_{x\in \bb{F}_{q}^{\times}}\pi[\mu](c_{2}(x)) + \dfrac{q}{2} \sum\limits_{\substack{z \in \bb{F}_{q^{2}}^\times\setminus \bb{F}_{q}^\times \\ z^{q+1}\in \{1\} \cup N}} \mu(c_{4}(z)) \notag \\
 &= -(q+1)\mu(\bb{F}_{q}^{\times})  - q\sum\limits_{\substack{z \in \bb{F}_{q^{2}}^\times\setminus \bb{F}_{q}^\times \\ z^{q+1}\in \{1\} \cup N}} \mu(z)\\
 &=  -(q+1)\mu(\bb{F}_{q}^{\times})  +q(1+\mu(-1)),
\end{align*}
with the last equality following from part (2) of Lemma~\ref{lem:gl2csums}. 

If $\mu(-1)=-1$, then $\bb{F}_{q}^{\times} \not\subseteq \mrm{Ker}(\mu)$, and thus, by orthogonality of characters, we have $\theta_{[\pi[\mu]]}=0$ in this case. Similarly, in the case $\mu(-1)=1$ and $\bb{F}_{q}^{\times} \not\subseteq \mrm{Ker}(\mu)$, we have 
 $\theta_{[\pi[\mu]]}=2q$. Finally, in the case $\bb{F}_{q}^{\times} \subseteq \mrm{Ker}(\mu)$,  we have $\mu(-1)=1$ and $\mu(\bb{F}_{q}^{\times})=(q-1)$.
 
 Since $\pi[\mu](t)/\pi[\mu](1)=\mu(-1)$, $\theta_{\pi[\mu]} \in \Phi^{+}$ if and only if $\mu(-1)=1$. The congruence results now follow by using the facts that $q$ is odd and $q^{2}-1 \equiv 0 \pmod{4}$.
\end{proof}
Now, we consider the eigenvalues indexed by the characters in the principal series.
\begin{lemma}\label{lem:pseries}
Let $\lambda_{1},\lambda_{2}$ be two distinct linear characters of $G$. Then,
\[ \theta_{I[\lambda_{1},\lambda_{2}]} = \begin{cases}  q(\lambda_{1}(-1)+ \lambda_{2}(-1)) + (q-1)^{2} & \text{if $\lambda_{1}^{-1}= \lambda_{2}$} \\ 
q(\lambda_{1}(-1)+ \lambda_{2}(-1)) & \text{otherwise.}
\end{cases}\]
Moreover,
\[ \theta_{I[\lambda_{1},\lambda_{2}]} \equiv \begin{cases} 2 \pmod{4} & \text{if $\theta_{I[\lambda_{1},\lambda_{2}]} \in \Phi^{+}$} \\ 
0 \pmod{4} & \text{if $\theta_{{I[\lambda_{1},\lambda_{2}]}} \in \Phi^{-}$.}
\end{cases}\]
\end{lemma}
\begin{proof}
  Let $\lambda_{1},\lambda_{2}$ be two distinct linear characters of $G$ and consider the character $I[\lambda_{1},\lambda_{2}]$. We have 
\begin{align*}
    \theta_{I[\lambda_{1},\lambda_{2}]} &= {{I[\lambda_{1},\lambda_{2}](\mfk{S})}}/{I[\lambda_{1},\lambda_{2}](1)} \\
    & = q(\lambda_{1}(-1)+ \lambda_{2}(-1)) + (q-1) \sum\limits_{x \in \bb{F}_{q}^{\times}} \lambda_{1}(x)\lambda_{2}(x).
\end{align*}
By orthogonality of characters, 

\[\sum\limits_{x \in \bb{F}_{q}^{\times}} \lambda_{1}(x)\lambda_{2}(x)= \begin{cases} (q-1) & \text{if $\lambda_{1}^{-1}=\lambda_{2}$,} \\
0 & \text{otherwise.}
\end{cases} \]
This proves the first part of the result. Moreover, we also have
\[\theta_{I[\lambda_{1},\lambda_{2}]} \equiv \lambda_{1}(-1)+ \lambda_{2}(-1) \pmod{4}.\]
We observe that $\theta_{I[\lambda_{1},\lambda_{2}]} \in \Phi^{+}$ if and only if $\lambda_{1}(-1)\lambda_{2}(-1)=1$. Now, the second part of the result follows. 
\end{proof}

Lemma~\ref{lem:gl2inte}, Lemma~\ref{lem:lin}, Lemma~\ref{lem:st}, Lemma~\ref{lem:cusp}, and Lemma~\ref{lem:pseries} show that the graph $\Gamma$ satisfies the premise of Corollary~\ref{cor:pstccs} for $a=2$. As $\Gamma$ is a regular graph with its highest eigenvalue being of multiplicity one, it is connected. We have now proven the following result.
\begin{theorem}\label{thm:gl2}
Let $q$ be an odd prime power and let $G=\gl(2,q)$ and $t:= \begin{bmatrix}
    -1 & 0 \\
    0 & -1
\end{bmatrix}$. Let $N$ denote the set of non-squares in $\bb{F}_{q}^{\times}$ and define $\mfk{S}$ to be the set of elements which are conjugate to an element of the set
\[\{c_{3}(1,-1)\} \bigcup \left(\bigcup\limits_{x \in \bb{F}_{q}^{\times}}c_{2}(x) \right) \bigcup \left( \bigcup\limits_{\substack{z \in \bb{F}_{q^{2}}^\times\setminus \bb{F}_{q}^\times \\ z^{q+1}\in \{1\} \cup N}} c_{4}(z)\right).\] Then, a perfect state transfer occurs at some time $t$, between vertices $x$ and $xt$ of the connected graph $\Gamma:=\mrm{Cay}(G,\mfk{S})$, for all $x\in G$. \qed
\end{theorem}

We remark that when $q=3$, the set $\mfk{S}=\gl(2,3) \setminus \{I,t\}$, and thus, $\Gamma$ is the complement of the disjoint union of $24$ copies of $K_{2}$. As we remarked in the introduction, this is a trivial example, and we desire to obtain a more ``interesting'' example. Let $\mfk{T}$ be the set of all non-central elements of orders $2$, $3$, $4$, and $6$. The graph $\Gamma'$ belongs to the conjugacy class scheme on $\gl(2,3)$ and a spanning subgraph of $\Gamma$. It is elementary to check that $\Gamma'$ is connected and that it satisfies the premise of Corollary~\ref{cor:pstccs}, and thus that it admits PST. When $q>3$, the graph $\Gamma$ is not the complement of the disjoint union of copies of $K_{2}$. This proves Theorem~\ref{thm:cay} for $\gl(2,q)$.

\subsection{PST in $\GU(2,q)$}\label{sub:gupst}
Let $q$ be an odd prime power. In this section, we construct a Cayley graph on $G:=\GU(2,q)$, which possesses PST. We note that $t:= \begin{bmatrix}
    -1 & 0 \\
    0 & -1
\end{bmatrix} $ is the unique central involution of $G$. As in the previous two subsections, we make use of Corollary~\ref{cor:pstccs} to construct a graph belonging to the conjugacy class scheme on $\GU(2,q)$, which exhibits PST. 

As defined in \S~\ref{sec:gu2c}, $E$ is the unique subgroup of index $q-1$ in $\bb{F}_{q^{2}}^{\times}$. Let $Q$ be the unique index-$2$ subgroup of $E$, and let $R:=E\setminus Q$. For convenience, we denote  $\bb{F}_{q}^{\times}$ by $F$.

\[\mfk{S} := C_{3}(1, -1) \bigcup \left( \bigcup\limits_{x \in E}C_{2}(x) \right) \bigcup \left( \bigcup\limits_{\substack{z \in \bb{F}_{q^{2}}^{\times} \setminus E\\ z^{q-1} \in \{1\} \cup R }}C_{4}(z) \right).\]

We will show that $\Gamma:= \mrm{Cay}(G,\mfk{S})$ is a graph that exhibits PST. The proof of this is essentially the same as the proof of Theorem~\ref{thm:gl2}.

First, we note that $\mfk{S}$ is a union of rational conjugacy classes of $G$. Arguing as in the proof Lemma~\ref{lem:gl2inte}, it follows that the eigenvalues of $\Gamma$ are integral. We first establish some auxiliary results. Using orthogonality of characters as we did in the proof of Lemma~\ref{lem:square}, we obtain the following result.

\begin{lemma}\label{lem:squareu}
Let $\lambda$ be a linear character of $E$.
\[\sum\limits_{x \in E} \lambda(x^{2})= 2\lambda(Q) = \begin{cases}
    q+1 & \text{if $Q \subset \mrm{Ker}(\lambda)$} \\
    0 & \text{otherwise.}
\end{cases} \]    \qed
\end{lemma}

The map $z \mapsto z^{1-q}$ from $\bb{F}_{q^{2}}^{\times}$ to $E$, is a surjective homomorphism with $F$ as its kernel. Given $x \in E$, the set $F_{x}:=\{ z \in\bb{F}_{q^{2}}^{\times}\ :\ z^{1-q} =x \}$, is a coset of $F$. Also, given $z \in E$, we have $z^{-q}=z$, and thus, we have $F_{x} \cap E \neq \emptyset$ if and only if $x\in Q $.  Moreover, we have
 \begin{subequations}\label{eq:Fcoset}
 \begin{align}
     |F_{x} \cap E| &= \begin{cases} 2 & \text{if $x\in Q$} \\
     0 & \text{otherwise.}
     \end{cases} \\
     E \cap F & = \{1,-1\}.
  \end{align}   
 \end{subequations}
Therefore, we have 

\begin{equation}\label{eq:set4u}
    \{z \in \bb{F}_{q^{2}}^{\times}\setminus E \ : \ z^{1-q}\in \{1\} \cup R\} =   \left(\bigcup\limits_{x \in R} F_{x} \right) \cup (F\setminus \{\pm 1\}). 
\end{equation}
Given a character $\mu$ of $\bb{F}_{q^{2}}^{\times}$, with $\mu \neq \mu^{-q}$, we have $\mu|_{F}$ is non-trivial and hence $\mu(F)=0$. Using the above equation and arguing as in the proof of Lemma~\ref{lem:gl2csums}, we obtain:

\begin{lemma}\label{lem:gu2csums}
    Let $\lambda$ be a non-trivial linear character of $E$ and let $\mu$ be linear character of $\bb{F}_{q^{2}}^{\times}$ with $\mu \neq \mu^{-q}$. 
    Then, we have 
    \begin{enumerate}
        \item $\sum\limits_{\substack{z \in \bb{F}_{q^{2}}^{\times} \setminus E \\ z^{1-q}\in \{1\} \cup R}} \lambda(z^{1-q})= (q-3) -(q-1)\lambda(Q)$ 
        \item $\sum\limits_{\substack{z \in \bb{F}_{q^{2}}^{\times} \setminus E \\ z^{1-q}\in \{1\} \cup R}}\mu(z) = -(1+ \mu(-1))$
    \end{enumerate}
\end{lemma}

For any linear character $\lambda$, since $\lambda(t)/\lambda(1)=1$, we have $\theta_{\lambda} \in \Phi^{+}$. We also have
\begin{align*}
\theta_{\lambda} & = \lambda(\mfk{S})\\
&= q(q-1)\lambda(-1) +(q^{2}-1)\sum\limits_{x \in E} \lambda(x^2) + \dfrac{q(q+1)}{2} \sum\limits_{\substack{z \in \bb{F}_{q^{2}}^{\times} \setminus E \\ z^{1-q}\in \{1\} \cup R}} \lambda(z^{1-q}) \lambda(z^{1-q}) 
\end{align*}

Using Lemma~\ref{lem:squareu} and Lemma~\ref{lem:gu2csums}, we have the following:

\begin{lemma}\label{lem:linu}
If $\lambda$ is a linear character of $G$, then
\[\theta_{\lambda}=  \begin{cases} q(q-1)\lambda(-1) + \dfrac{q(q+1)(q-3)}{2} & \text{if $Q \not\subseteq \mrm{Ker}(\lambda)$} \\
 q(q-1)\lambda(-1) +(q^{2}-1)(q+1) +\dfrac{q(q+1)(q-3)}{2} - \dfrac{q(q+1)(q-3)(q-1)}{4} & \text{if $Q=\mrm{Ker}(\lambda)$} \\
 q(q-1)\lambda(-1) +(q^{2}-1)(q+1)+ \dfrac{q(q+1)(q-3)}{2} +\dfrac{q(q+1)(q-3)(q-1)}{4} & \text{if $\lambda=1$.}
\end{cases}\] Moreover, $\theta_{\lambda} \in \Phi^{+}$ and $\theta_{\lambda} \equiv 2 \pmod{4}$.
\end{lemma}

Similar computations on $S_{\lambda}$ yield the following result.

\begin{lemma}\label{lem:stu}
If $\lambda$ is a linear character of $G$, then
\[\theta_{S_{\lambda}}=  \begin{cases} (q-1)\lambda(-1) + \dfrac{(q+1)(q-3)}{2} & \text{if $Q \not\subseteq \mrm{Ker}(\lambda)$} \\
 (q-1)\lambda(-1) +\dfrac{(q+1)(q-3)}{2} - \dfrac{(q+1)(q-3)(q-1)}{4} & \text{if $Q=\mrm{Ker}(\lambda)$} \\
 (q-1)\lambda(-1) + \dfrac{(q+1)(q-3)}{2} +\dfrac{(q+1)(q-3)(q-1)}{4} & \text{if $\lambda=1$.}
\end{cases}\] Moreover, $\theta_{\lambda} \in \Phi^{+}$ and $\theta_{\lambda} \equiv 2 \pmod{4}$.
\end{lemma}

Now, we consider $\theta_{\pi[\mu]}$, where $\mu$ is a linear character of $\bb{F}_{q^{2}}^{\times}$ with $\mu \neq \mu^{-q}$. Using the values of $\pi[\mu]$ given in Table~\ref{tab:guct} and part (2) of Lemma~\ref{lem:gu2csums}, we arrive at the following: 
\[\theta_{\pi[\mu]} = (q-1)\mu(E)-q(1+\mu(-1)).\] Since $\pi[\mu](t)=(q+1)\mu(-1)$, we see that $\theta_{\pi[\mu]} \in \Phi^{+}$ if and only if $\mu(-1)=1$. We have thus proven the following result.

\begin{lemma}\label{lem:cuspu}
Let $\mu$ be a linear character of $\bb{F}_{q^{2}}^{\times}$ with $\mu \neq \mu^{-q}$. Then,
\[ \theta_{\pi[\mu]} = \begin{cases} 0  & \text{if $\mu(-1)=-1$,} \\
 -2q  & \text{if $\mu(-1)=1$ and $E \not\subseteq \mrm{Ker}(\mu)$,}\\
(q^{2}-1) -2q  & \text{if $E\subseteq \mrm{Ker}(\mu)$.}
\end{cases}\]
Moreover, 
\[\theta_{\pi[\mu]} \equiv \begin{cases} 
0 \pmod{4} & \text{if $\theta_{\pi[\mu]} \in \Phi^{-}$,} \\
2 \pmod{4} & \text{if $\theta_{\pi[\mu]} \in \Phi^{+}$.} \\
\end{cases} \]
\end{lemma}

Given distinct linear characters $\lambda_{1}$ and $\lambda_{2}$ of $E$, we consider $\theta_{I[\lambda_{1},\lambda_{2}]}$. We have 
\[\theta_{I[\lambda_{1},\lambda_{2}]} =-q(\lambda_{1}(-1) + \lambda_{2}(-1))-(q+1) \sum\limits_{x\in E} \lambda_{1}(x) \lambda_{2}(x).\]
Since $I[\lambda_{1}, \lambda_{2}](t)/(q-1)= \lambda_{1}(-1) \lambda_{2}(-1)$, we have $\theta_{I[\lambda_{1},\lambda_{2}]} \in \Phi^{+}$ if and only if $\lambda_{1}(-1) \lambda_{2}(-1)=1$. Using orthogonality of characters, we arrive at the following result.         \begin{lemma}\label{lem:pseriesu}
Let $\lambda_{1},\lambda_{2}$ be two distinct linear characters of $G$. Then,
\[ \theta_{I[\lambda_{1},\lambda_{2}]} = \begin{cases}  -q(\lambda_{1}(-1)+ \lambda_{2}(-1)) - (q+1)^{2} & \text{if $\lambda_{1}^{-1}= \lambda_{2}$} \\ 
-q(\lambda_{1}(-1)+ \lambda_{2}(-1)) & \text{otherwise.}
\end{cases}\]
Moreover,
\[ \theta_{I[\lambda_{1},\lambda_{2}]} \equiv \begin{cases} 2 \pmod{4} & \text{if $\theta_{I[\lambda_{1},\lambda_{2}]} \in \Phi^{+}$} \\ 
0 \pmod{4} & \text{if $\theta_{{I[\lambda_{1},\lambda_{2}]}} \in \Phi^{-}$.}
\end{cases}\]
\end{lemma}

Lemma~\ref{lem:linu}, Lemma~\ref{lem:stu}, Lemma~\ref{lem:cuspu}, and Lemma~\ref{lem:pseriesu} show that the graph $\Gamma$ satisfies the premise of Corollary~\ref{cor:pstccs} for $a=2$. As $\Gamma$ is a regular graph with its highest eigenvalue being of multiplicity one, it is connected. We have now proven the following result.
\begin{theorem}\label{thm:gu2}
Let $q$ be an odd prime power and let $G=\GU(2,q)$ and $t:= \begin{bmatrix}
    -1 & 0 \\
    0 & -1
\end{bmatrix}$. Let $Q$ denote the set of non-squares in $E$, and define \[\mfk{S} := C_{3}(1, -1) \bigcup \left( \bigcup\limits_{x \in E}C_{2}(x) \right) \bigcup \left( \bigcup\limits_{\substack{z \in \bb{F}_{q^{2}}^{\times} \setminus E\\ z^{q-1} \in \{1\} \cup R }}C_{4}(z) \right).\]
 Then, a perfect state transfer occurs at some time $t$, between vertices $x$ and $xt$ of the connected graph $\Gamma:=\mrm{Cay}(G,\mfk{S})$, for all $x\in G$. \qed
\end{theorem}

Since $|E|\geq 4$, there exists $x \in E \setminus \{1,-1\}$ and thus $C_{3}(1,x) \notin \mfk{S}$. Therefore, $\mfk{S} \subsetneqq G \setminus \{I,t\}$, and thus $\ov{\Gamma}$ is not a disjoint union of copies of $K_{2}$. We can now conclude that Theorem~\ref{thm:cay} is true in the case of $\GU(2, q)$. 
\subsection{PST in $\sli(2,q)$.}\label{sub:slpst}
Let $q$ be an odd prime power and let $G:=\sli(2,q)$ be the special linear group of dimension $2$ on $\bb{F}_{q}$. In this section, we construct a connected Cayley graph on $G$, which possesses PST. The group contains a unique central involution, $t:= -I$. Let $\mfk{S}$ be the collection of all elements of $G$, which are conjugate to one of the elements of $\{d_{2}(\pm 1, 1),d_{2}(\pm 1, \Delta)\}$. Here, $d_{2}(\pm 1, 1),d_{2}(\pm 1, \Delta)$ are as defined in \S~\ref{sec:sl2cc}. We consider the graph $\Gamma_{\sli}:=\mrm{Cay}(G,\{t\} \cup \mfk{S})$. Arguing as in the proof of Lemma~\ref{lem:gl2inte}, it follows that $\Gamma_{\sli}$ has integral eigenvalues. We now apply Corollary~\ref{cor:pstccs} to show that there is PST between vertices $x$ and $xt$ of this graph, for all $x \in G$.

We first observe that for all $\chi \in \irr(G)$, we have 
\begin{equation}\label{eq:tchi-1}
\theta_{\chi} := \dfrac{\chi(t)}{\chi(1)} + \dfrac{q^{2}-1}{2 \chi(1)} r_{\chi},
\end{equation}
where \[r_{\chi} =\left[ \chi(d_{2}(1,1)) + \chi(d_{2}(-1,1)) + \chi(d_{2}(1,\Delta)) + \chi(d_{2}(-1,\Delta)) \right].\]
As discussed in \S~\ref{sec:sl2cc}, any $\chi \in \irr(G) \setminus \{U^{\pm},Z^{\pm}\}$ is a restriction of an irreducible character of $\gl(2,q)$. Using Table~\ref{tab:gl2qct}, we observe that for $\chi \in \irr(G) \setminus \{U^{\pm},Z^{\pm}\}$, $r_{\chi}$ must be divisible by $4$. From Table~\ref{tab:sl}, we observe that, for any $\chi \in \{U^{\pm},Z^{\pm}\}$, $r_{\chi}$ is even.
We also note that either $\chi(1) \mid \dfrac{q^{2}-1}{2} $ or $r_{\chi}=0$. Moreover, in the case $\chi \in \{U^{\pm},Z^{\pm}\}$, we have $\chi(1) \mid \dfrac{q^{2}-1}{4}$. These observations along with \eqref{eq:tchi-1} yield 
\[\theta_{\chi} \equiv \dfrac{\chi(t)}{\chi(1)}  \pmod{4},\] for all $\chi \in \irr(G)$.An application of Corollary~\ref{cor:pstccs} yields the following.

\begin{theorem}\label{thm:pstsl}
    Let $q$ be an odd prime power $p$ and let $G=\sli(2,q)$ and $t:= \begin{bmatrix}
    -1 & 0 \\
    0 & -1
\end{bmatrix}$. Let $\mfk{S}$ be the collection of all elements of $G$ of orders $p$ and $2p$. 

Then, a perfect state transfer occurs at some time $t$, between vertices $x$ and $xt$ of the connected graph $\Gamma:=\mrm{Cay}(G,\{t\} \cup \mfk{S})$, for all $x\in G$. \qed
\end{theorem}

\begin{remark} We remark that the $X:=\mrm{Cay}(\gl(2,q),\{t\} \cup \mfk{S})$ is a graph belonging to the conjugacy scheme on $\gl(2,q)$ and it satisfies all the conditions in Corollary~\ref{cor:pstccs}. However, this graph is  disconnected --- it is the
disjoint union of $q-1$ copies of the graph $\Gamma$ in the above theorem. For this reason, we do not consider this graph in the case of $\gl(2,q)$.
\end{remark}

\section{Proof of Theorem~\ref{thm:orb}}\label{sec:orbproof}
As in the premise of Theorem~\ref{thm:orb}, let $q$ be a prime power satisfying $q\equiv 3\pmod{4}$. For ease of notation, we set $G:=\gl(2,q^{2})$ and $H:=\gl(2,q)$. We shall apply
the results of \S~\ref{sec:repandccgl} to the group $G$ rather than to
$H$, and so it is important to note that $q$ must be replaced by $q^2$. We set $B:=B(q^2)$.

We describe a graph belonging to $G\sslash H$ which admits PST. 

\begin{con}\label{con:ospst} 
Let $z \in Z(G)\setminus H$ be such that $z^{2}H=H$. Given $x,y \in \F_{q^{2}}^{\times}$, define $m_{x,y}:= \begin{bmatrix}
x & 0 \\
0 & y
\end{bmatrix}$. Define $\Gamma_{q}$ to be the graph whose vertex set is $G/H$ in which $rH$ is adjacent to $sH$ if and only if $r^{-1}s \in \left( HzH \cup \bigcup\limits_{x,y \in \F_{q^{2}}^{\times}\ \&\ x\neq y} Hm_{x,y}H\right)$.
\end{con}

We will show that $\Gamma_{q}$ admits PST. To check this, we need to compute its  spectrum. We first gather some useful results about $G\sslash H$.
Let $F:G \to G$ denote the map defined by $[x_{ij}] \mapsto [x_{ij}^{q}]$. By Theorem~3.4 of \cite{gow}, for any $x,y \in G$, we have $HxH=HyH$ if and only if $x^{-1}F(x)$ is conjugate to $y^{-1}F(y)$.
Using Theorem~3.4 of \cite{gow}, it follows that $Hm_{x,y}H=Hm_{a,b}H$ if and only if $\{x\F_{q}^{\times}, y\F_{q}^{\times}\}=\{a\F_{q}^{\times}, b\F_{q}^{\times}\}$. Let $\Delta$ be a complete set of coset representatives for $\F_{q}^{\times}$ in $\F_{q^{2}}^{\times}$.  For ease of notation, for all $x, y \in \F_{q^{2}}^{\times}$, we denote $A_{m_{x,y}}$ by $A_{x,y}$. The adjacency matrix of $\Gamma_{q}$ is
\begin{equation}\label{eq:gammaqadjacency}
A := A_{z} + \dfrac{1}{2}\sum\limits_{\{(x,y) \in \Delta \times \Delta\ :\ x\neq y\} } A_{x,y}   
\end{equation}      

We shall now apply  Corollary~\ref{cor:ospst} to our graph, with $z$ in the role of $t$ in that theorem.
For $\chi\in\irr({G\sslash H})$, let $\theta_\chi$ be the eigenvalue of $A$ defined by $A\mathbf E_\chi=\theta_\chi\mathbf E_\chi$. 
Note that the degree of $\Gamma_{q}$ is $\theta_1$ and that $\theta_\chi$s are not necessarily distinct.
We have $\theta(zH)/|H| =\pm1$, and $\theta_{\chi} \in\Phi^{\pm}$ with the same sign.

 From \eqref{eq:osct}, we recall that $A_{z}\mbf{E}_{\theta} =\theta(zH)/|H| \mbf{E}_{\theta}$. Suppose that it can be established that every eigenvalue of $A-A_z$ is an integer divisible by 4.
Then $\theta_\chi\in\Phi^+$ if and only if $\theta_\chi\equiv 1\pmod 4$
and  $\theta_\chi\in\Phi^-$ if and only if $\theta_\chi\equiv 3\pmod 4$. Then, by Corollary~\ref{cor:ospst} we have PST between any pair of vertices related by $A_z$.
%Then since $\theta_1\in\Phi^+$,  we have $v_2(\theta_1-\theta_\chi)=1$ 
%for all $\theta_\chi\in\Phi^-$ and $v_2(\theta_1-\theta_\chi)>1$
%for all $\theta_\chi\in\Phi^+$. Thus the hypotheses of  Theorem~\ref{PSTAssoc} and we have PST between any pair of vertices
%related by $A_z$.
Therefore, to show that $\Gamma_{q}$ admits PST, we are reduced to proving the following proposition.

\begin{proposition}\label{prop:0m4evalue}
Every eigenvalue of  
\[D:=\dfrac{1}{2}\sum\limits_{\{(x,y) \in \Delta \times \Delta\ :\ x\neq y\} } A_{x,y}\] is an integer and is divisible by $4$.
\end{proposition}

As $D$ is in the Bose--Mesner algebra $\bb{C}[G\sslash H]$, any eigenspace of $D$ is spanned by columns of a collection of matrices in $\{\mb{E}_{\chi}\ :\ \chi \in \irr(G\sslash H)\}$. 
To compute the eigenvalues of $D$ we use the formula \eqref{eq:osct}. As is evident from the formula, it is beneficial to  compute character  sums of the form $\chi(gH)$, where $g \in \{m_{x,y}\ : (x,y) \in \Delta^{2}\ \&\ x\neq y\}$.
\subsection{Character sums over cosets of $H$.} 
We recall that $B$ is the subgroup of upper triangular matrices in $G$.   
Let $T$
be the subgroup of diagonal matrices and $U$ be the subgroup of unipotent matrices. Note that $T\cong \F_{q^{2}}^{\times} \times \F_{q^{2}}^{\times}$ and that $B/U \cong T$. Thus, any character of $T$ extends to a character of $B$. Given $\theta \in \widehat{T}$, let $I[\theta]:=\mrm{Ind}_{H}^{G}(\theta)$. We will now compute $I[\theta](gH)$. 

We start by fixing a matrix representation of $G$ that affords $I[\theta]$ as its character. 

Define 
$\omega_\beta:=\begin{bmatrix}
1 & 0 \\
\beta & 1
\end{bmatrix}$ for $\beta\in \F_{q^{2}}$ and $\omega_\infty:=\begin{bmatrix}
0 & 1 \\
1 & 0
\end{bmatrix}$.  Let $\PP=\F_{q^{2}}\cup \{\infty\}$ be the projective line over $\F_{q^{2}}$
The set 
\[
\Omega :=\left\{ \omega_\beta\ :\ \beta \in \PP \right\} 
\]
is a complete set of left coset representatives for $B$ in $G$. 

 Given $g \in G$ and $\beta \in \PP$, the elements
 $\sigma_g(\beta)\in \PP$  and $b(g,\beta)\in B$ are uniquely
 defined by the equation
 \begin{equation}\label{cosetaction}
 g\omega_\beta=\omega_{\sigma_g(\beta)}b(g,\beta).
 \end{equation}
The map $g \to \sigma_{g}$ defines a transitive action of $G$ on $\PP$, with $B$ as the stabilizer of $0$.

 Now given $\theta \in \widehat{T}$, let $P_{\theta}(g)$ to be the $|\PP| \times |\PP|$ matrix satisfying 
 \[P_{\theta}(g)_{\beta, \alpha } = \delta_{\sigma_{g}(\alpha),\beta} \theta(b(g, \alpha)),\] 
 for all $\alpha,\beta \in \PP$. The map $P: G \to \gl(|\PP|,\bb{C})$ is a representation affording $I[\theta]$ as its character. 
 We observe that 
\[I[\theta](gH)=\mrm{Tr}\left(P_{\theta}(g)P_{\theta}(H)\right).\] 

We first start computing $M_{\theta}:=P_{\theta}(H)$. Given $\alpha$, $\beta \in \PP$, define $H_{\alpha \to \beta} :=\{h\in H \ :\ \sigma_{h}(\alpha) =\beta \}$.  Using the definition of $P_{\theta}$, we see that 

\begin{equation}\label{eq:initial}
M_{\theta}(\beta, \alpha) = \sum\limits_{h \in H_{\alpha\to \beta}} \theta(b(h,{\alpha})).
\end{equation} 

Thus, if $\alpha, \beta \in \PP$ are not in the same $H$-orbit, we have  $M_{\theta}(\beta,\alpha)=0$. We consider action of $H$ on $\PP$. Note that \[\bigcup\limits_{\beta \in \F_{q}} \omega_\beta B \cup \omega_{\infty} B \subset HB.\] As $|HB|= \dfrac{|H||B|}{|\bo(q)|}$,  (where $\bo(q)=H \cap B$ is the subgroup of upper triangular matrices in $H$), we see that 
\[\bigcup\limits_{\beta \in \F_{q}} \omega_\beta B \cup \omega_{\infty}  B = HB,\] and thus the $H$-orbit containing $0$ is 

$\PP_1:=\F_q\cup\{\infty\}$. Given $\alpha$, $\beta\in \F_{q^{2}} \setminus \F_{q}$, there is a unique  $(c_{\alpha, \beta}, d_{\alpha, \beta}) \in \F^{2}_{q}$, with 
$d_{\alpha, \beta}\neq 0$, 
such that $\beta= c_{\alpha, \beta} + d_{\alpha, \beta} \alpha$. 
Let $h_{\alpha,\beta}:=\begin{bmatrix}
    1 & 0\\
    c_{\alpha, \beta} & d_{\alpha, \beta}
\end{bmatrix}$. Then we have

            \begin{equation}\label{elementaction} h_{\alpha, \beta}\omega_\alpha=\omega_\beta\begin{bmatrix}1&0\\0&d_{\alpha,\beta}\end{bmatrix},
            \end{equation}
so $\sigma_{h_{\alpha, \beta}}(\alpha)= \beta$ and $b(h_{\alpha, \beta}, \alpha)=\begin{bmatrix}1&0\\0&d_{\alpha, \beta}\end{bmatrix}$. This shows that $H$ is transitive on $\PP_{2}:=\PP\setminus \PP_{1}$. We now derive formulae for entries of $M_{\theta}$. We define $C$ to be the subgroup $\left\{ \begin{bmatrix}
a & 0 \\
0 & a^q
\end{bmatrix}\ :\ a \in \F_{q^{2}}^{\times} \right\}$ of $B$. 
          
\begin{lemma}\label{lem:charsumH}
Given $\alpha,\beta \in \PP$, we have 
\[M_{\theta}(\beta,\alpha) = \begin{cases}
\theta (\bo(q)) & \text{if $\alpha, \beta \in \F_{q} \cup \{\infty\}$,} \\
\theta\left(\begin{bmatrix}
1 & 0 \\
0 & d_{\alpha,\beta}
\end{bmatrix} \right) \theta(C) & \text{if $\alpha, \beta \in \F_{q^{2}} \setminus \F_{q}$,} \\
 0 & \text{otherwise.}
\end{cases}  \]
\end{lemma}
\begin{proof}
The equality in the last case follows from \eqref{eq:initial}. Assume $\alpha, \beta $ are in the same $H$-orbit. From \eqref{eq:initial}, we have 
\begin{align*}
M_{\theta}(\beta,\alpha) &= \sum\limits_{h \in H_{{\alpha} \to {\beta}}} \theta(b(h,{\alpha})).  
\end{align*}

Now consider the set $S :=\{b(h,{\alpha})\ :\ h \in H_{{\alpha}\to \beta}\}$. Given $x,y \in H_{\alpha\to \beta}$, we have $b(x,\alpha)={\omega_\beta}^{-1}x \omega_\alpha$ and $b(y,\alpha)={\omega_\beta}^{-1}y \omega_\alpha$, and thus, 
\begin{equation}\label{eq:b}
b(x,\alpha)^{-1}b(y,\alpha) ={\omega_\alpha}^{-1}x^{-1}y \omega_\alpha \in (H_{\alpha\to \alpha})^{\omega_\alpha}. 
\end{equation}
Fix an element $z \in H_{\alpha \to \beta}$. Using \eqref{eq:b}, we have \[S \subset b(z,\alpha)  (H_{\alpha\to \alpha})^{\omega_\alpha}.\] Again from \eqref{eq:b}, we have $b(x,\alpha)=b(y,\alpha)$ if and only if $x=y$. Thus $|S|=|H_{{\alpha\to \alpha}}|$, and thus

\[S = b(z,\alpha)  (H_{\alpha\to \alpha})^{\omega_\alpha}.\]
 
 This shows that
 
 \begin{align}\label{eq:general}
 M_{\theta}(\beta,\alpha) &= \theta(S) \notag \\
 &= \theta(b(z,\alpha)) \theta(  (H_{\alpha\to \alpha})^{\omega_\alpha})
 \end{align}

Case 1: Assume that $\alpha, \beta \in \F_{q} \cup \{\infty\}$. In this case, $\omega_\alpha$ and $\omega_\beta$ lie in $H$. Thus, for any $h \in H_{\alpha \to \beta}$, we have $b(h, \alpha)= {\omega_\beta}^{-1} h \omega_\alpha \in H$, and therefore $S \subset H$. We also have $(H_{\alpha\to \alpha})^{\omega_\alpha}= b(z,\alpha)^{-1}S \in H$. Since $(H_{\alpha\to \alpha})^{\omega_\alpha} \subset B$, we have $(H_{\alpha\to \alpha})^{\omega_\alpha} \subset B\cap H= \bo(q)$. As $H_{\alpha\to \alpha}$ is the stabilizer of $\alpha$ with respect to the transitive action of $H$ on $\PP_{1}$, we have $|H_{\alpha\to \alpha}|=|\bo(q)|$. Since $b(z,\alpha) \in B \cap H =\bo(q)$, we have $S= \bo(q)$. This proves the first equality in the lemma.
 
Case 2: Assume that $\alpha, \beta \in \F_{q^{2}} \setminus \F_{q}$. Setting $z=h_{\alpha,\beta}$ in \eqref{eq:general}, we have 
 \[M_{\theta}(\beta,\alpha)= \theta\left(\begin{bmatrix}
1 & 0 \\
0 & d_{\alpha,\beta}
\end{bmatrix} \right) \theta(  (H_{\alpha\to \alpha})^{\omega_\alpha}) \]

Let $s,t \in \F_{q}$ be such that $\alpha^{2}=s+t\alpha$. We note that $B$ is the set of elements fixing the projective point $0 \in \bb{P}$. Therefore, given $x:=\begin{bmatrix}
a & b\\
c & d
\end{bmatrix} \in H$, we have $x \in H_{\alpha \to \alpha}$ if and only if $x^{\omega_\alpha} \in B$. Thus, $x \in H_{\alpha \to \alpha}$ if and only if 

 \[\begin{bmatrix}
 a+ b \alpha & b \\
 c-a\alpha +d \alpha -b \alpha^{2} & d-b\alpha
 \end{bmatrix} = \begin{bmatrix}
 a+ b \alpha & b \\
 (c-bs)+(d-a-bt)\alpha & d-b\alpha 
 \end{bmatrix} \in B.\]
 
 As $\{1,\alpha\}$ is a basis for $\F_{q^{2}}$, we conclude that $x \in H_{\alpha \to \alpha}$ if and only if (i) $c=bs$, (ii) $d=a+bt$, and (iii) $(a,b) \neq (0, 0)$. We note that (i) and (ii) ensure that $x^{\omega_\alpha}$ is an upper triangular matrix, and (iii) ensures that $x^{\omega_\alpha}$ is invertible. We have 
 \[(H_{\alpha \to \alpha})^{\omega_\alpha} = \left\{ \begin{bmatrix}
 a +b \alpha & b \\
 0 & (a+bt)-b\alpha
\end{bmatrix} \ :\ a,b \in \F_{q}\ \&\ (a,b) \neq (0,0)  \right\}.\] 

We have
\[\theta\left((H_{\alpha \to \alpha})^{\omega_\alpha} \right) = \theta \left( \left\{ (a+b\alpha,(a+bt)-b\alpha)\ :\ a,b \in \F_{q}\ \&\ (a,b) \neq (0,0) \right\} \right)\]

The trace and determinant of each matrix in $(H_{\alpha \to \alpha})^{\omega_\alpha}$ is in $\F_{q}$. Therefore, given $b\neq 0$, every pair of the form $(a+b\alpha,(a+bt)-b\alpha)$ forms the solution set of a degree $2$ irreducible polynomial over $\F_{q}$. Since roots of degree $2$ irreducible polynomials are conjugate pairs, we can now conclude that
\[\left\{ (a+b\alpha,(a+bt)-b\alpha)\ :\ a,b \in \F_{q}\ \&\ (a,b) \neq (0,0) \right\}  = \{ (z, z^{q})\ :\ z \in \F_{q^{2}}^{\times} \},\] and thus that
\[\theta\left((H_{\alpha \to \alpha})^{\omega_\alpha} \right) =\theta(C).\] This concludes the proof.        
\end{proof}
Using the above lemma and $I[\theta](gH)=\mrm{Tr}\left(P_{\theta}(g)M_{\theta}\right)$, we obtain the following result. 

\begin{cor}\label{cor:chsums}
\[I[{\theta}](g H) = \sum\limits_{\{\beta \in \PP_{1} \ :\ \sigma_{g}(\beta) \in \PP_{1}\}} \theta(b(g,\beta))\theta(\bo(q)) + \sum\limits_{\{\beta \in \PP_{2} \ :\ \sigma_{g}(\beta) \in \PP_{2}\}} \theta(b(g, \beta)) \theta\left(\begin{bmatrix}
1 & 0 \\
0 & d_{ \sigma_{g}(\beta), \beta}
\end{bmatrix} \right) \theta(C).\]
\end{cor}

Given $x,y \in \F_{q^{2}}^{\times}$ with $x^{-1}y \notin \F_{q}^{\times}$, we now compute $I[{\theta}](m_{x,y} H)$ (recall that $m_{x,y}= \mrm{diag}(x,y)$). As $\theta$ is a linear character of $B$, it is irreducible when restricted to both $\bo(q)$ and $C$. From the above corollary, it follows that $I[\theta](gH)=0$ unless either $\bo(q) \subset \mrm{Ker}(\theta)$ or $C \subset \mrm{Ker}(\theta)$. Let us now assume that $\theta$ is such that either $\bo(q) \subset \mrm{Ker}(\theta)$ or $C \subset \mrm{Ker}(\theta)$. Let $\theta_{1},\theta_{2} \in \widehat{\F_{q^{2}}^{\times}}$ such that $\theta=[\theta_{1},\theta_{2}]$. We note that $\sigma_{m_{x,y}}\left( \beta \right) = x^{-1}y \beta$ and $b(m_{x,y}, \beta)=m_{x,y}$, for all $\beta \in \F_{q^{2}}$; and that $\sigma_{m_{x,y}}\left( \infty \right)= \infty$ and $b(m_{x,y}, \infty)=m_{y,x}$. Therefore,
\begin{align*}
\{\beta \in \PP_{1} \ :\ \sigma_{g}(\beta) \in \PP_{1}\} &= \left\{0,\infty   \right\} \\
\{\beta \in \PP_{2} \ :\ \sigma_{g}(\beta) \in \PP_{2}\} & = \F_{q^{2}}^\times \setminus \left(\F_{q}^\times \cup y^{-1}x\F_{q}^{\times} \right).
\end{align*}
These observations along with Corollary~\ref{cor:chsums} yield \begin{equation}\label{eq:dintchsum}
\begin{split}
I[\theta](m_{x,y}H) =  {} & (\theta_{1}(x)\theta_{2}(y)+\theta_{1}(y)\theta_{2}(x)) \theta(\bo(q)) \\{} & + \theta_{1}(x)\theta_{2}(y)\theta(C) \left[\sum\limits_{\alpha \in \F_{q^{2}}^\times \setminus \left(\F_{q}^\times \cup y^{-1}x\F_{q}^{\times} \right)} \theta_{2}(d_{ x^{-1}y\alpha,\alpha}) \right].
\end{split}
\end{equation}
We now simplify this sum.
\begin{lemma}
Let $x,y \in\F_{q^{2}}^{\times}$ with $x^{-1}y \notin \F_{q}^{\times}$ and let $\chi \in \widehat{\F_{q^{2}}^{\times}}$, then 
\[\sum\limits_{\alpha \in \F_{q^{2}}^\times \setminus (\left(\F_{q}^\times\cup y^{-1}x\F_{q}^{\times} \right)} \chi(d_{ x^{-1}y\alpha,\alpha})=  (q-1)\chi(\F_{q^{\times}}).\] 
\end{lemma}
\begin{proof}
We first recall some notation: given $\epsilon, \gamma \in \F_{q^{2}} \setminus \F_{q}$, $(c_{\epsilon,\gamma}, d_{\epsilon,\gamma}) \in \F_{q}\times \bb{F}_{q}^{\times}$ is the unique pair satisfying  
$\gamma=c_{\epsilon,\gamma}+ d_{\epsilon,\gamma}\epsilon$.   We can rewrite this equation as
\begin{equation}\label{eq:eps}
    \epsilon = -c_{\epsilon,\gamma} d_{\epsilon,\gamma}^{-1} + d_{\epsilon,\gamma}^{-1}\gamma,
\end{equation}
and thus, 
\begin{equation}\label{eq:getoeg}
c_{\gamma,\epsilon}= -c_{\epsilon,\gamma}d_{\epsilon,\gamma}^{-1}\quad\text{and}\quad d_{\gamma,\epsilon}= d_{\epsilon,\gamma}^{-1} .
\end{equation}

Set $\beta:=x^{-1}y$,  $r:=c_{\beta,\beta^{2}}$, and $s:=d_{\beta,\beta^{2}}$. By design $\beta \in \F_{q^{2}}^{\times}\setminus \F_{q}^{\times}$. Consider $\alpha \in \bb{F}_{q^{2}}^{\times} \setminus \left(\bb{F}_{q}^{\times} \cup \beta^{-1} \bb{F}_{q}^{\times} \right)$. By \eqref{eq:eps}, we have 
\begin{equation}\label{eq:atobtoa}
\beta = -c_{\beta,\alpha}d_{\beta,\alpha}^{-1} +d_{\beta,\alpha}^{-1}\alpha.
\end{equation}
Starting from $\alpha = c_{\beta,\alpha} + d_{\beta,\alpha} \beta$, we obtain
\begin{align*}
    \beta \alpha &= c_{\beta,\alpha} \beta + d_{\beta,\alpha} \beta^{2} \\
    & =  d_{\beta,\alpha}r+ (c_{\beta,\alpha}+sd_{\beta,\alpha}) \beta \qquad\text{(since $\beta^{2}=r+s\beta$)} \\
     &= (d_{\beta, \alpha}r- c^{2}_{\beta, \alpha}d_{\beta, \alpha}^{-1}+c_{\beta, \alpha}s) +(c_{\beta, \alpha}d_{\beta, \alpha}^{-1}+s) \alpha,
\end{align*}
with the last step following from \eqref{eq:atobtoa}. We have now shown that $d_{\alpha,\beta\alpha} = (c_{\beta, \alpha}d_{\beta,\alpha}^{-1}+s)$. Since $\alpha\notin \bb{F}_{q}^{\times} \cup \beta^{-1} \bb{F}_{q}^{\times}$, we have $\beta \alpha, \alpha \in \bb{F}_{q^{2}}^{\times} \setminus \bb{F}_{q}^{\times}$, and thus $d_{\alpha,\beta\alpha} \neq 0$. Now, using \eqref{eq:getoeg}, we have  
\begin{equation}\label{eq:batoa}
d_{\beta\alpha,\alpha} = d_{\alpha,\beta\alpha}^{-1} = (c_{\beta, \alpha}d_{\beta, \alpha}^{-1}+s)^{-1}.
\end{equation}

We claim that the function $\alpha \mapsto d_{\beta\alpha,\alpha}$ from $\bb{F}_{q^{2}}^{\times}\setminus \left(\bb{F}_{q}^{\times} \cup \beta^{-1} \bb{F}_{q}^{\times} \right)$ to $\bb{F}_{q}^{\times}$, is (i) constant on left cosets (in $\bb{F}_{q^{2}}^{\times}$) of $\bb{F}_{q}^{\times}$ that sit in the domain and (ii) takes distinct values on distinct $\bb{F}_{q}^{\times}$-left cosets in the domain. As the domain $\bb{F}_{q^{2}}^{\times}\setminus \left(\bb{F}_{q}^{\times} \cup \beta^{-1} \bb{F}_{q}^{\times} \right)$ partitions into $(q-1)$ distinct $\bb{F}_{q}^{\times}$-cosets in $\bb{F}_{q^{2}}^{\times}$, the result follows from the claim.

Given $t \in \bb{F}_{q}^{\times}$, we observe that $tc_{\beta, \alpha}=c_{\beta, t\alpha}$ and $td_{\beta,\alpha}=d_{\beta, t\alpha}$. Thus, by \eqref{eq:batoa}, for any $\alpha \in \bb{F}_{q^{2}}^{\times}\setminus \left(\bb{F}_{q}^{\times} \cup \beta^{-1} \bb{F}_{q}^{\times} \right)$, we have $d_{\beta \alpha,\alpha}=d_{\beta t\alpha, t\alpha}$.  Conversely, suppose that  $\alpha_{1},\alpha_{2} \in \bb{F}_{q^{2}}^{\times}\setminus \left(\bb{F}_{q}^{\times} \cup \beta^{-1} \bb{F}_{q}^{\times} \right)$ with $d_{\beta\alpha_{1},\alpha_{1}}=d_{\beta\alpha_{2}, \alpha_{2}}$. Then, using \eqref{eq:batoa}, we have 
 $c_{\beta, \alpha_{1}}d_{\beta, \alpha_{1}}^{-1}=c_{\beta, \alpha_{2}}d_{\beta, \alpha_{2}}^{-1}$.  
 Setting $t=d_{\beta, \alpha_{1}}d_{\beta, \alpha_{2}}^{-1}$, we have 
 \[\alpha_{1}=c_{\beta, \alpha_{1}} + d_{\beta, \alpha_{1}} \beta = tc_{\beta, \alpha_{2}} +td_{\beta,\alpha_{2}} \beta =t\alpha_{2},\] and thus, $\alpha_{1} \in \alpha_{2}\bb{F}_{q}^{\times}$.
 This proves the claim in the above paragraph, which concludes our proof.
\end{proof}

Using \eqref{eq:dintchsum} and the above result, we proved the following.

\begin{cor}\label{cor:charsumdiag}
Let $x,y \in \F_{q^{2}}^{\times}$ with $x^{-1}y \notin \F_{q}$, and let $\theta \in \widehat{B}$. Then 
\[I[\theta](m_{x,y}H) = (\theta_{1}(x)\theta_{2}(y)+\theta_{1}(y)\theta_{2}(x)) \theta(\bo(q)) + {(q-1)}\theta_{1}(x)\theta_{2}(y)\theta(C) \theta_{2}(\F_{q}^{\times})\]
\end{cor}
\subsection{Eigenvalues of $D$.}
We now describe all the eigenvalues of $D$. Let $\chi \in \irr(G\sslash H)$. Using \eqref{eq:osct} and the definition of $D$, we see that 

\begin{equation*}
\mcal{E}_{\chi} := \dfrac{1}{2} \sum\limits_{\{(x,y) \in \Delta \times \Delta\ :\ x\neq y\}} \dfrac{\chi(m_{x,y}H)}{|H \cap H^{m_{x,y}}|} 
\end{equation*}  
is an eigenvalue of $D$. Moreover $\{\mcal{E}_{\chi}\ :\ \chi \in \irr(G\sslash H)\}$ is the complete set of eigenvalues of $D$. For technical reasons, we extend the definition of $\mcal{E}_{\chi}$ to all complex characters of $G$.

Given a complex character $\chi$ of $G$, we define
\begin{equation}\label{eq:echi}
\mcal{E}_{\chi} := \dfrac{1}{2} \sum\limits_{\{(x,y) \in \Delta \times \Delta\ :\ x\neq y\}} \dfrac{\chi(m_{x,y}H)}{|H \cap H^{m_{x,y}}|} 
\end{equation}

For each $x \in \F_{q}$, $c_{1}(x)\in H$, and $H$ contains exactly $q-1$ elements conjugate to $c_{2}(x)$. All the other elements are conjugate to an element of the form $c_{3}(x,y)$ for some $(x,y) \in \F^{\times}_{q^{2}} \times \F^{\times}_{q^{2}}$ with $x \neq y$. This shows that $\left\langle \pi[\mu]|_{H},\one \right\rangle=0$, for every cuspidal character $\pi[\mu]$ of $G$. Thus no cuspidal character is in $\irr(G \sslash H)$.

Now, we figure out which parabolic characters are in $\irr(G \sslash H)$. Recall that a parabolic character is an irreducible summand of a character obtained by inducing a linear character of $B$. Consider a linear character $\theta=[\theta_{1},\theta_{2}]$ (here $\theta_{1},\theta_{2} \in \F_{q^{2}}^{\times}$ and $\theta$ is as defined in \eqref{eq:linchpar}) of $B$. We now consider $I[\theta]:=\mrm{Ind}_{B}^{G}(\theta)$. We have 
\begin{align*}
\left\langle \one, I[\theta]|_{H} \right\rangle &= \mrm{Trace}(M_{\theta}) \\
& = \dfrac{(q+1)\theta(\bo(q))}{|H|} + \dfrac{(q^{2}-q)\theta(C)}{|H|} \ \text{using Lemma~\ref{lem:charsumH}}).
\end{align*} 

Therefore, $\left\langle \one, I[\theta]|_{H} \right\rangle\neq 0$ if and only if either $\bo(q) \subset \mrm{Ker}(\theta)$ or $C \subset \mrm{Ker}(\theta)$. We observe that \\
(i) $\bo(q) \subset \mrm{Ker}(\theta)$ if and only if $\F_{q}^{\times} \subset \mrm{Ker}(\theta_{1})$ (equivalently $\theta_{1}^{q+1}=\one$) and $\F_{q}^{\times} \subset \mrm{Ker}(\theta_{2})$ (equivalently $\theta_{2}^{q+1}=\one$); and\\
(ii) $C \subset \mrm{Ker}(\theta)$, if and only if $\theta_{1}=\theta_{2}^{-q}$.  \\

We now conclude that

\[\left\langle \one, I[\theta]|_{H} \right\rangle = \begin{cases}
1 & \text{if $\theta_{1} \neq \theta_{2}$ and $\theta_{1}^{q+1}=\one= \theta_{2}^{q+1}$}; \\
1 & \text{if $\theta_{1} \neq \theta_{2}$ and $\theta_{1}=\one= \theta_{2}^{-q}$}; \\
2 & \text{if $\theta_{1} = \theta_{2}$ and $\theta_{1}^{q+1}=\one= \theta_{2}^{q+1}$}; \\
0 & \text{otherwise.}
\end{cases} \]

Let $\theta_{\lambda}:=[\lambda,\lambda]$, for some $\lambda \in \widehat{\F_{q^{2}}^{\times}}$ such that $\lambda^{q+1}=1$. We recall that $I[\theta_{\lambda}]=\lambda + S_{\lambda}$ is the decomposition of $I[\theta_{\lambda}]$ into irreducible summands. Thus, given  $\lambda \in \widehat{\F_{q^{2}}^{\times}}$, we have $\lambda \in \irr(G\sslash H)$ (respectively $S_{\lambda} \in \irr(G\sslash H)$) if and only if $\lambda^{q+1}=1$. Recalling that $\{\mcal{E}_{\chi}\ :\ \chi \in \irr(G\sslash H)\}$ is the complete set of eigenvalues of $D$, we have now proved the following result.

\begin{lemma}\label{lem:bevalue}
Any eigenvalue of $D$ is an element of 
\[\{\mcal{E}_{I[\theta]}\ :\ \bo(q) \subset \mrm{Ker}(\theta)\ \text{or}\ C \subset \mrm{Ker}(\theta)\} \cup \{\mcal{E}_{\lambda},\mcal{E}_{S_{\lambda}}\ :\ \lambda \in \widehat{\F_{q^{2}}^{\times}}\ \&\ \lambda^{q+1}=\one\}.\]
\end{lemma}

We now compute $\mcal{E}_{I[\theta]}$. We observe that $|H \cap H^{m_{x,y}}|=(q-1)^{2}$. Using this observation and Corollary~\ref{cor:charsumdiag}, we have 
\begin{align}\label{eq:Etheta}
\mcal{E}_{I[\theta]} &= \dfrac{1}{2(q-1)^{2}} \sum\limits_{\{(x,y) \in \Delta \times \Delta\ :\ x\neq y\}} (\theta_{1}(x)\theta_{2}(y)+\theta_{1}(y)\theta_{2}(x)) \theta(\bo(q)) + (q-1)\theta_{1}(x)\theta_{2}(y)\theta(C) \theta_{2}(\F_{q}^{\times}) \notag \\
&= \left[\dfrac{\theta(\bo(q))}{(q-1)^{2}} + \dfrac{\theta(C)\theta_{2}(\F_{q}^{\times})}{2(q-1)} \right] \times \sum\limits_{\{(x,y) \in \Delta \times \Delta\ :\ x\neq y\}} \theta_{1}(x)\theta_{2}(y)
\end{align}
We are now ready to prove Proposition~\ref{prop:0m4evalue}. We do this in a series of lemmas.

\begin{lemma}
For all $\theta \in \widehat{B}$, $\mcal{E}_{I[\theta]}$ is an integer divisible by $4$.
\end{lemma}

\begin{proof}

Let $\theta=[\theta_{1}, \theta_{2}]$
We recall that $\theta$ is a linear character of $\bo(q)$ and $C$ and since $\theta_{2}$ is a linear character of $\F_{q}^{\times}$. Therefore, $\left[\dfrac{\theta(\bo(q))}{(q-1)^{2}} + \dfrac{\theta(C)\theta_{2}(\F_{q}^{\times})}{2(q-1)} \right]$  is an element in 
$\{0, q, \dfrac{(q^{2}-1)}{2}, q+ \dfrac{(q^{2}-1)}{2}\}$, and thus an integer. 

Case 1: First, we consider the case in which $\mrm{Ker}(\theta) \supset \bo(q)$, that is, $\F_{q}^{\times}$ is in the kernel of both $\theta_{1}$ and $\theta_{2}$. We have
\begin{align*}
\sum\limits_{\{(x,y) \in \Delta \times \Delta\ :\ x\neq y\}} \theta(m_{x,y}) &= 
\sum\limits_{\{(x,y) \in \Delta \times \Delta\}} \theta(m_{x,y}) -\sum\limits_{x \in \Delta} \theta(x,x)  \\
&=\sum\limits_{\{(x,y) \in \Delta \times \Delta\}} \theta_{1}(x) \theta_{2}(y) -\sum\limits_{x \in \Delta} \theta_{1}(x)\theta_{2}(x)  \\
&= \dfrac{1}{(q-1)^{2}} \left(\sum\limits_{\{w \in \F_{q^{2}}^{\times} \times \F_{q^{2}}^{\times}\}} \theta(w) \right) -\dfrac{1}{q-1} \left(\sum\limits_{x \in \F_{q^{2}}^{\times}} \theta_{1}(x)\theta_{2}(x) \right),
\end{align*}
with the last equality following from the fact that $\F_{q}^{\times}$ is in the kernel of both $\theta_{1}$ and $\theta_{2}$. As $\theta$ is an irreducible representation of $\F_{q^{2}}^{\times} \times \F_{q^{2}}^{\times}$, the first sum is either $0$ or $(q+1)^{2}$. As $\theta_{1}$ and $\theta_{2}$ are irreducible characters of $\F_{q^{2}}^{\times}$, using orthogonality of characters, the second sum is either $q+1$ or $0$. Thus, if $\mrm{Ker}(\theta) \supset \bo(q)$, then  $\sum\limits_{\{(x,y) \in \Delta \times \Delta\ :\ x\neq y\}} \theta(m_{x,y})$ is an integer divisible by $q+1$. Thus, in this case, $q+1 \mid \mcal{E}_{I[\theta]}$. Since $q\equiv 3 \pmod{4}$, this implies that $4 \mid \mcal{E}_{I[\theta]}$.

Case 2 :Now, we consider the case $C \subset \mrm{Ker}(\theta)$. This condition is satisfied if and only if $\theta_{1}=\theta_{2}^{-q}$. Thus, we have $\theta_{1}=\theta_{2}^{-q}$. \\
Subcase 1: If $\F_{q}^{\times} \not\subset \mrm{Ker}(\theta_{2})$, then  $\bo(q) \not\subset \mrm{Ker}(\theta)$. Thus, in this case, $\left[\dfrac{\theta(\bo(q))}{(q-1)^{2}} + \dfrac{\theta(C)\theta_{2}(\F_{q}^{\times})}{2(q-1)} \right]=0$, and thus, $\mcal{E}_{I[\theta]}=0$.\\
Subcase 2: If $\F_{q}^{\times} \subset \mrm{Ker}(\theta_{2})$, since $\theta_{1}=\theta_{2}^{-q}$, we must have $\F_{q}^{\times} \subset \mrm{Ker}(\theta_{1})$. Therefore, in this case, $\bo(q) \subset \mrm{Ker}(\theta)$. From Case 1, it follows that 
$4 \mid \mcal{E}_{I[\theta]}$.
\end{proof}

We now turn our attention to linear characters in $\irr(G\sslash H)$.
\begin{lemma}
For all $\lambda \in \widehat{G}\cong \widehat{\F_{q^{2}}^{\times}} $ with $\lambda^{q+1}=\one$, $\mcal{E}_{\lambda}$ is an integer divisible by $4$.
\end{lemma}
\begin{proof}
Using \eqref{eq:echi} and $|H \cap H^{m_{x,y}}|=(q-1)^{2}$, we have 
\[\mcal{E}_{\lambda} = \dfrac{\lambda(H)}{2(q-1)^{2}} \left(\sum\limits_{\{(x,y) \in \Delta \times \Delta\ :\ x\neq y\}} \lambda(m_{x,y}) \right).\] Since $\lambda^{q+1}=\one$, we have $\F_{q}^{\times}\subset \mrm{Ker}(\lambda)$, and therefore, $\lambda|H= \one$. Thus, we have 
\begin{align*}
\mcal{E}_{\lambda} &= \dfrac{|H|}{2(q-1)^{2}} \left(\sum\limits_{\{(x,y) \in \Delta \times \Delta\ :\ x\neq y\}} \lambda(m_{x,y}) \right) \\
&= \dfrac{q(q+1)}{2} \left(\sum\limits_{\{(x,y) \in \Delta \times \Delta\ :\ x\neq y\}} \lambda(m_{x,y}) \right).
\end{align*}
Since $\F_{q}^{\times}\subset \mrm{Ker}(\lambda)$, we have 
$\lambda(m_{x,y})=\lambda_{m_{tx,ty}}$, for all $t \in \F_{q}$. Therefore,
\begin{align*}
\mcal{E}_{\lambda} &= \dfrac{q(q+1)}{2} \left(\sum\limits_{\{(x,y) \in \F_{q^{2}}^{\times} \times \F_{q^{2}}^{\times} \ :\ x\F_{q}^{\times}\neq y \F_{q}^{\times}\}} \lambda(m_{x,y}) \right) \\
&= \dfrac{q(q+1)}{2} \left(\sum\limits_{\{(x,y) \in \F_{q^{2}}^{\times} \times \F_{q^{2}}^{\times} \}} [\lambda,\lambda](x,y) - \sum\limits_{x \in \F_{q^{2}}^{\times}} \lambda(x^{2})\right)
\end{align*}
Let $S$ denote the subgroup of squares in $\F_{q^{2}}^{\times}$, then we have
\[\mcal{E}_{\lambda} = \dfrac{q(q+1)}{2} \left( [\lambda,\lambda]\left(\F_{q^{2}}^{\times} \times \F_{q^{2}}^{\times}\right)- 2 \lambda(S)\right).\]
Since $[\lambda, \lambda]$ is an irreducible character of $\F_{q^{2}}^{\times} \times \F_{q^{2}}^{\times}$, and since $\lambda$ restricts to an irreducible character of $S$, we have $2\lambda(S) \in \{0,q^{2}-1\}$ and $[\lambda, \lambda]\left(\F_{q^{2}}^{\times} \times \F_{q^{2}}^{\times}\right) \in \{0,(q^{2}-1)^{2}\}$. The result now follows.
\end{proof}

\begin{lemma}
For all $\lambda \in \widehat{\F_{q^{2}}^{\times}} $ with $\lambda^{q+1}=\one$, $\mcal{E}_{S_{\lambda}}$ is an integer divisible by $4$.
\end{lemma}
\begin{proof}
Since $I[\lambda,\lambda]=\lambda + S_{\lambda}$, using \eqref{eq:echi}, we see that 
\[\mcal{E}_{S_{\lambda}}= \mcal{E}_{I[\lambda, \lambda]}-\mcal{E}_{{\lambda}}.\] Using the previous two results, we can conclude that the RHS of the above is $0\pmod{4}$, and thus the result follows.
\end{proof}

Proposition~\ref{prop:0m4evalue} now follows using Lemma~\ref{lem:bevalue}.
\bibliographystyle{plain}
\bibliography{ref}                        

\begin{thebibliography}{10}

\bibitem{Ahmadi}
Bahman Ahmadi, M.~H. Shirdareh~Haghighi, and Ahmad Mokhtar.
\newblock Perfect quantum state transfer on the {J}ohnson scheme.
\newblock {\em Linear Algebra Appl.}, 584:326--342, 2020.

\bibitem{Angeles-Canul}
Ricardo~J. Angeles-Canul, Rachael~M. Norton, Michael~C. Opperman,
  Christopher~C. Paribello, Matthew~C. Russell, and Christino Tamon.
\newblock Perfect state transfer, integral circulants, and join of graphs.
\newblock {\em Quantum Inf. Comput.}, 10, 2010.

\bibitem{ArnadottirGodsil2}
Arnbj\"org~Soff\'ia \'Arnad\'ottir and Chris Godsil.
\newblock On state transfer in cayley graphs for abelian groups.
\newblock {\em Quantum Inf. Process.}, 22(8), 2023.

\bibitem{BannaiIto}
Eiichi Bannai and Tatsuro Ito.
\newblock {\em Algebraic combinatorics. {I}}.
\newblock The Benjamin/Cummings Publishing Co., Inc., Menlo Park, CA, 1984.
\newblock Association schemes.

\bibitem{Basic}
M.~Ba\v{s}i\'{c}.
\newblock Characterization of circulant networks having perfect state transfer.
\newblock {\em Quantum Inf. Process.}, 12(1):345--364, 2013.

\bibitem{BernasconiGS}
A.~Bernasconi, C.~Godsil, and S.~Severini.
\newblock Quantum networks on cubelike graphs.
\newblock {\em Phys. Rev. A}, 78(052320), 2008.

\bibitem{CaoFeng}
X.~Cao and K~Feng.
\newblock Perfect state transfer on cayley graphs over dihedral groups.
\newblock {\em Linear and Multilinear Algebra}, 69:343--360, 2019.

\bibitem{ChanUMCubes}
A.~Chan.
\newblock Complex hadamard matrices, instantaneous uniform mixing and cubes.
\newblock {\em Algebraic Combinatorics}, 3(3):757--774, 2020.

\bibitem{CheungGodsil}
W.~Cheung and C.~Godsil.
\newblock Perfect state transfer in cubelike graphs.
\newblock {\em Linear Algebra Appl.}, 435:2468--2474, 2011.

\bibitem{Christandl}
Matthias Christandl, Nilanjana Datta, Artur Ekert, and Andrew~J. Landahl.
\newblock Perfect state transfer in quantum spin networks.
\newblock {\em Phys. Rev. Lett.}, 92:187902, May 2004.

\bibitem{CoutinhoGodsil}
G.~Coutinho and C.~Godsil.
\newblock Perfect state transfer in products and covers of graphs.
\newblock {\em Linear and Multilinear Algebra}, 64(2):235--246, 2015.

\bibitem{CoutinhoGodsilBook}
G.~Coutinho and C.~Godsil.
\newblock {\em Graph Spectra and Continuous Quantum Walks}.
\newblock U. Waterloo, 2021.

\bibitem{CoutinhoJulianoSpier}
G.~Coutinho, E.~Juliano, and T.~Spier.
\newblock No perfect state transfer in trees with more than three vertices.
\newblock 2023.
\newblock arxiv:2305.10199v1.

\bibitem{coutinhothesis}
Gabriel Coutinho.
\newblock {\em Quantum state transfer in graphs}.
\newblock PhD thesis, University of Waterloo, 2014.

\bibitem{coutinho2015perfect}
Gabriel Coutinho, Chris Godsil, Krystal Guo, and Fr{\'e}d{\'e}ric Vanhove.
\newblock Perfect state transfer on distance-regular graphs and association
  schemes.
\newblock {\em Linear Algebra and its Applications}, 478:108--130, 2015.

\bibitem{Ennola}
V.~Ennola.
\newblock On the characters of the finite unitary groups.
\newblock {\em Ann. Acad. Sci. Fenn. Ser. A I.}, 323:3--35, 1962.

\bibitem{ennolaconjugacy}
Veikko Ennola.
\newblock On the conjugacy classes of the finite unitary groups.
\newblock {\em Annales Fennici Mathematici}, 313, 1962.

\bibitem{fulton2013representation}
William Fulton and Joe Harris.
\newblock {\em Representation theory: a first course}, volume 129.
\newblock Springer Science \& Business Media, 2013.

\bibitem{GodsilStateTransfer}
C.~Godsil.
\newblock State transfer on graphs.
\newblock {\em Discrete Mathematics}, 312(1):129--147, 2012.

\bibitem{GMbook}
Chris Godsil and Karen Meagher.
\newblock {\em {E}rd{\H{o}}s--{K}o--{R}ado {T}heorems: {A}lgebraic
  {A}pproaches}, volume 149 of {\em Cambridge Studies in Advanced Mathematics}.
\newblock Cambridge University Press, Cambridge, 2016.

\bibitem{gow}
Roderick Gow.
\newblock Two multiplicity-free permutation representations of the general
  linear group {${\rm GL}(n,q^2)$}.
\newblock {\em Math. Z.}, 188(1):45--54, 1984.

\bibitem{jordan1907group}
Herbert~E Jordan.
\newblock Group-characters of various types of linear groups.
\newblock {\em American Journal of Mathematics}, 29(4):387--405, 1907.

\bibitem{Kay}
A.~Kay.
\newblock Basics of perfect communication through quantum networks.
\newblock {\em Phys. Rev. A}, 84(022337), 2011.

\bibitem{Schur1907}
J.~Schur.
\newblock Untersuchungen über die darstellung der endlichen gruppen durch
  gebrochene lineare substitutionen.
\newblock {\em Journal für die reine und angewandte Mathematik}, 132:85--137,
  1907.

\bibitem{SorciSin}
Peter Sin and Julien Sorci.
\newblock Continuous-time quantum walks on {C}ayley graphs of extraspecial
  groups.
\newblock {\em Algebr. Comb.}, 5(4):699--714, 2022.

\bibitem{TanFengCao}
Y.~Tan, K.~Feng, and X.~Cao.
\newblock Perfect state transfer on abelian cayley graphs.
\newblock {\em Linear Algebra Appl.}, 563:331--352, 2019.

\end{thebibliography}
\end{document}